\documentclass{amsart}
\usepackage[latin1]{inputenc}
\usepackage[english]{babel}
\usepackage[T1]{fontenc}
\usepackage{amsmath}
\usepackage{appendix}
\usepackage{amssymb}
\usepackage{mathtools}
\usepackage{amsthm}
\usepackage{array,booktabs}
\usepackage{graphicx}
\usepackage{comment}
\usepackage{lmodern}
\usepackage{quoting}
\usepackage{enumitem}
\usepackage{physics}
\usepackage{xfrac}
\def\modulo{\text{ \rm mod }}
\usepackage{mathtools}
\usepackage{csquotes}
\usepackage{multirow}
\usepackage{cite}
\usepackage{tikz}
\usepackage{tikz-cd}
\usetikzlibrary{matrix,arrows,decorations.pathmorphing}
\usepackage{wrapfig}
\renewcommand{\Re}{\operatorname{Re}}
\renewcommand{\Im}{\operatorname{Im}}
\def\R{\ensuremath\mathbb{R}}
\def\C{\ensuremath\mathbb{C}}
\def\Z{\ensuremath\mathbb{Z}}
\def\Q{\ensuremath\mathbb{Q}}
\def\N{\ensuremath\mathbb{N}}
\def\H{\ensuremath\mathbb{H}}
\def\F{\ensuremath\mathbb{F}}
\def\K{\ensuremath\mathbb{K}}
\usepackage{hyperref}
\usepackage[final]{pdfpages}
\usepackage{verbatim}
\newtheorem{thm}{Theorem}[section]
\newtheorem{defi}[thm]{Definition}
\newtheorem{cor}[thm]{Corollary}
\newtheorem{lemma}[thm]{Lemma}
\newtheorem{prop}[thm]{Proposition}
\theoremstyle{remark}
\newtheorem{remark}[thm]{Remark}

\def\eps{\ensuremath\varepsilon}
\def\0{\emptyset}

\def\SL{\mathrm{SL}}
\def\Hom{\mathrm{Hom}}

\def\SO{\mathrm{SO}}

\def\GL{\mathrm{GL}}
\def\vol{\mathrm{vol}}
\numberwithin{equation}{section}

\newcommand{\inprod}[2]{\left \langle  {#1} , {#2} \right \rangle}
\newcommand{\gag}{\Gamma_{\infty}' \backslash \Gamma / \Gamma_{\infty}'}
\newcommand{\lr}[1]{\left (   {#1} \right )}
\newcommand{\Ltwo}{L^2(\Gamma \backslash \H^{n+1}, \chi)}
\def\a{\ensuremath\mathfrak{a}}
\def\sa{\ensuremath{\sigma_{\mathfrak{a}}}}
\def\Gi{\ensuremath{\Gamma_{\infty}'}}
\def\GH{\ensuremath{\Gamma \backslash \H^{n+1}}}
\usepackage{accents,amsmath,amsfonts,amsthm,amssymb,graphicx,mathrsfs,mathtools}
\usepackage[stretch=10]{microtype}
\usepackage{xcolor}
\definecolor{dark-red}{rgb}{0.4,0.15,0.15}
\definecolor{dark-blue}{rgb}{0.15,0.15,0.4}
\definecolor{medium-blue}{rgb}{0,0,0.5}

\hypersetup{
	pdftitle={Residual distribution of modular symbols},
	pdfauthor={Petru Constantinescu and Asbj\o{}rn Christian Nordentoft},
	pdfnewwindow=true,
	colorlinks, linkcolor={dark-red},
	citecolor={dark-blue}, urlcolor={medium-blue}
}

\begin{document}
\title[Residual distribution of modular symbols]{Residual equidistribution of modular symbols and cohomology classes for quotients of hyperbolic $n$-space}
\author{Petru Constantinescu}

\address{Department of Mathematics, University College London,
	25 Gordon Street, London, UK, WC1H  0AY}

\email{\href{mailto:petru.constantinescu.17@ucl.ac.uk}{petru.constantinescu.17@ucl.ac.uk}}

\author{Asbj\o{}rn Christian Nordentoft}

\address{Mathematical Institute of the University of Bonn, Endenicher Allee 60, Bonn 53115, Germany }

\email{\href{mailto:acnordentoft@outlook.com}{acnordentoft@outlook.com}}

\date{\today}
\begin{abstract}
We provide a new and simple automorphic method using Eisenstein series to study the equidistribution of modular symbols modulo primes, which we apply to prove an average version of a conjecture of Mazur and Rubin. More precisely, we prove that modular symbols corresponding to a Hecke basis of weight 2 cusp forms are asymptotically jointly equidistributed mod $p$ while we allow restrictions on the location of the cusps. As an application, we obtain a residual equidistribution result for Dedekind sums. Furthermore, we calculate the variance of the distribution and show a surprising bias with connections to perturbation theory. Additionally, we prove the full conjecture in some particular cases using a connection to Eisenstein congruences. Finally, our methods generalise to equidistribution results for cohomology classes of finite volume quotients of $n$-dimensional hyperbolic space.

\end{abstract}

\maketitle
\section{Introduction}
Modular symbols are certain periods of weight 2 cusp forms introduced by Birch and Manin and they are an indispensable tool for studying (twisted) $L$-functions of holomorphic cusp forms \cite{Man72}, \cite{Mazur79} and for computing modular forms \cite{Cremona97}. Modular symbols define elements of certain cohomology groups and the results of this paper thus fit into a bigger picture of the study of (co)homology of arithmetic groups, which has received a lot of attention recently \cite{BergeronVenkatesh13}, \cite{CalegariVenk19} due to their deep connections with number theory coming from \cite{Scholze15}.

Recently, Mazur and Rubin initiated the study of the arithmetic distribution of modular symbols and put forward a number of conjectures \cite{MaRu19}, which have received a lot of attention, see the work of Petridis--Risager\cite{PeRi}, Bettin--Drappeau \cite{BeDr19}, Blomer et al. \cite[Chapter 9]{BlFoKoMiMiSa}, Diamantis et al. \cite{DiHoKiLe20}, Lee--Sun \cite{LeeSun19}, Sun\cite{Sun20}, Nordentoft \cite{No19}, Constantinescu \cite{petru}. One of these conjectures (see  \cite{MaRu}) predicts that (normalised) modular symbols should equidistribute among the residue classes modulo $p$. Recently, an average version of this conjecture was settled by Lee and Sun \cite[Theorem I]{LeeSun19} using dynamical methods. In this paper we introduce a new automorphic method for studying the mod $p$ distribution of modular symbols, which also applies to more general cohomology classes. As is the case in \cite{LeeSun19}, we obtain an average version of the mod $p$ conjecture of Mazur and Rubin (and its generalisations), but with further refinements. Using different arguments, we can actually prove the full conjecture in some special cases (specific $p$ and specific cusp forms), see Section \ref{special case}.

Our automorphic methods enable us to deal with a much more general setup compared to the work of Lee and Sun and thus we obtain a number of new results:
\begin{enumerate}
    \item Firstly, we obtain {\it joint} equidistribution for the mod $p$ values of modular symbols (appropriately normalised) associated to a Hecke basis of weight 2 cusp forms restricted to cusps which lie in a {\it fixed} interval of $\R/\Z$.
    \item We calculate the variance of the distribution and show a surprising bias for large $p$.
    \item We show some particular cases of the full conjecture using connections with Eisenstein congruences.
    
    \item As an application of our method, we obtain a residual equidistribution result for Dedekind sums.
    
    \item Lastly, we extend the equidistribution results to classes in the cohomology of general finite volume quotients of higher dimensional hyperbolic spaces.
    
\end{enumerate}  
We note that in the case of higher dimensional hyperbolic spaces there is interesting torsion in the cohomology. The breakthrough of Scholze \cite{Scholze15} established that such torsion classes have associated Galois representations. This was actually our original motivation for studying the higher dimensional cases. Furthermore, Bergeron and Venkatesh \cite{BergeronVenkatesh13} have conjectured that, at least in the three dimensional case, there is an abundance of torsion in the relevant cohomology group. In this paper we are able to shed light on the distribution properties of these cohomology classes. In Appendix \ref{cohomology} we will survey what is known about the dimensions of the cohomology groups, which our results apply to.

\subsection{Results for modular symbols} 
Let us state the result in the simplest case for the two dimensional hyperbolic space in an arithmetic setup. We define the {\it modular symbol map} associated to a weight 2 and level $N$ cusp form $f\in \mathcal{S}_2(\Gamma_0(N))$ as the map 
\begin{equation}\label{modularsymbol} \Q\ni r \mapsto \langle r, f\rangle:=2\pi i \int_{r}^{i\infty} f(z)dz, \end{equation}
where the contour integral is taken along a vertical line. One way to think about this map is as the Poincar\'{e} pairing on $\Gamma_0(N)\backslash \H^2$ between the 1-form $2\pi i f(z)dz$ and the homology class of paths containing the geodesic from $r$ to $i\infty$.  
Now assume that $f$ is a Hecke-normalised eigenform. Then by  \cite[Sec. 1]{MaRu19}, there exist periods $\Omega_{f,+}$ and $\Omega_{f,-}$ such that for all $a/q\in \Q$ with $N|q$, we have $\mathfrak{m}^{\pm}_{f}(a/q)\in \Z$ with full image, where 
\begin{align}\label{oddevenMS} \mathfrak{m}^{\pm}_{f}(a/q):=\frac{1}{\Omega_{f,\pm}}\left(\langle a/q, f\rangle \pm \langle -a/q, f\rangle\right). \end{align}
 Given a basis of Hecke eigenforms $f_1,\ldots, f_d$ and a prime $p$, we can consider the map
$$r\mapsto \mathfrak{m}_{N,p}(r):=(\mathfrak{m}^+_{f_1}(r),\mathfrak{m}^-_{f_1}(r),\ldots ,\mathfrak{m}^-_{f_d}(r), r)\in(\Z/p\Z)^{2d}\times (\R/\Z) $$
as a random variable defined on the outcome space 
\begin{equation} \label{Omega}\Omega_{Q,N}:= \{a/q\mid 0< a<q\leq Q, (a,q)=1, N|q\}\end{equation}
endowed with the uniform probability measure. Then we have the following equidistribution result.
\begin{thm}
\label{mainthm1}
The random variables $\mathfrak{m}_{N,p}$ defined on the outcome spaces $\Omega_{Q,N}$ converge in distribution to the uniform distribution on $(\Z/p\Z)^{2d}\times (\R/\Z)$ as $Q\rightarrow \infty$. More precisely, for any fixed $\mathbf{a} \in (\Z / p \Z)^{2d}$ and any interval $I\subset \R/\Z$, we have
\begin{equation*}
\frac{ \# \left \{ a/q \in \Omega_{Q,N}\cap I  \mid  (\mathfrak{m}^+_{f_1}(a/q),\ldots ,\mathfrak{m}^-_{f_d}(a/q))  \equiv \mathbf{a} \modulo p \right \}}{\# \Omega_{Q,N}} =  \frac{|I|}{p^{2d}}+ o(1)
\end{equation*}
 as $Q \to \infty$. 
\end{thm}
\begin{remark}
Similarly, we can prove equidistribution modulo $p^n$ (see Theorem \ref{mainthmHn} below). This translates to the fact that the random variables $(\mathfrak{m}^+_{f_1},\ldots ,\mathfrak{m}^-_{f_d})$ considered as maps $\Omega_{Q,N}\cap I\rightarrow\Q_p$ are asymptotically distributed with respect to the (multivariate) standard $p$-adic Gau{\ss}ian (as defined in for instance \cite{Zelenov19}).  \end{remark}

The next natural question is to ask how well the values equidistribute. We answer this by studying the {\lq\lq}variance{\rq\rq} of the residual distribution modulo $p$ of the random variables $\mathfrak{m}_f^{\pm}$ on the sample space $\Omega_{Q,N}$. Furthermore, we show an analogue of Chebyshev's bias for large $p$, in the sense that the modular symbols are {\lq\lq}biased{\rq\rq} towards the residue class $0$ mod $p$. 

\begin{thm}
\label{variancethm}
For large enough $p$, there exist constants $c_p, \delta_p$>0 such that
\begin{equation*}
    \sum_{a \in \Z/ p\Z}\left ( \frac{ \# \{ b/q  \in \Omega_{Q,N} \mid \mathfrak{m}_{f}^\pm(b/q) \equiv a \modulo p \}}{\#{\Omega_{Q,N}}}  - \frac{1}{p}  \right )^2 \sim c_p Q^{-\delta_p}
\end{equation*}
as $Q \to \infty$. Moreover, as $p \to \infty$, we have that 
$ c_p=2/p+O(p^{-2})$ and $\delta_p\rightarrow 0$.

Furthermore, for $p$ large enough, we have for $Q$ large enough (depending on $p$) that:
\begin{align} \label{bias2}
\#\{ b/q \in \Omega_{Q,N}\mid \mathfrak{m}^\pm_f(b/q)\equiv a\modulo p\}\leq \#\{ b/q \in \Omega_{Q,N}\mid \mathfrak{m}^\pm_f(b/q)\equiv 0\modulo p\},  \end{align}
with equality exactly if $a\equiv 0\mod p$.
\end{thm}


\begin{remark}
We explicitly evaluate the constants $c_p$ and $\delta_p$ and moreover we obtain asymptotics for the deviation from the mean for different residue classes when $p$ is large, see Section \ref{variance section} for more details.
\end{remark}

We can also show that some specific cases of the conjecture of Mazur and Rubin hold, that is without taking an extra average. We state here the result in the simplest case and refer to Section \ref{special case} for the more general case.
\begin{thm}
\label{special case thm}
Let $f\in \mathcal{S}_2(\Gamma_0(11))$ be the unique cusp form of weight 2 and level $11$. Then the values of $\mathfrak{m}^+_{f}$ on $\{\frac{a}{q}\mid (a,q)=1, 0<a<q\}$ equidistribute exactly modulo $5$ for all $q\equiv 0\modulo 11$.
\end{thm}


As a consequence of the method developed to study the special case as in Theorem \ref{special case thm}, we deduce a residual equidistribution result for classical Dedekind sums, given by $$s(a,q):=\sum_{k=1}^q (\!(k/q)\!)(\!(ak/q)\!)$$ with $(\!()\!)$ the \emph{sawtooth function}. We allow for both an {\lq\lq}algebraic{\rq\rq} and {\lq\lq}archimedian{\rq\rq} restriction on $(a,q)$. Our result supplements the vast literature on the archimedean distributional properties of Dedekind sums, see \cite{Girstmair18}, \cite{Bruggeman94} for surveys of results.

\begin{cor}\label{dedekind}
Let $N,p\geq 5$ be primes such that $p| N-1$ and $H\leq (\Z/N\Z)^\times$ the unique subgroup of index $p$. Fix some class $a_0\in (\Z/N\Z)^\times$ and some interval $I\subset \R/\Z$. Then the values of 
$$s(a, Nq)-s(a,q)-\frac{(N-1)(a+\overline{a})}{12 q}$$
(where $\overline{a}a\equiv 1\modulo Nq$) on  the outcome space
$$\left\{ (a,q) \mid 0<q\leq Q,a\in (\Z/Nq\Z)^\times , a\in a_0 H, a/q\in I \right\}$$
are all $p$-integral and equidistribute mod $p$ as $Q\rightarrow \infty$. 
\end{cor}

We observe that the modular symbols map gives rise to a map $\Gamma_0(N)\rightarrow \C$ by putting $\langle\gamma, f \rangle:= \langle\gamma\infty, f \rangle$, where $\gamma\infty=a/c$ with $a,c$ the left upper and lower entries of $\gamma\in \Gamma_0(N)$. By shifting the contour and doing a change of variable we see that
\begin{align*}
\langle\gamma_1\gamma_2, f \rangle=\langle\gamma_1, f \rangle+2\pi i\int_{\gamma_1\infty}^{\gamma_1\gamma_2\infty}f(z)dz=\langle\gamma_1, f \rangle+\langle\gamma_2, f \rangle,\end{align*}
which shows that modular symbols define an additive character on $\Gamma_0(N)$ and thus an element of (the cuspidal part of) the cohomology group $H^1(\Gamma_0(N), \C)$. Furthermore, by the integrality conditions, we see that the normalised modular symbols $\mathfrak{m}_{f,p}^\pm$ define elements of $H^1(\Gamma_0(N), \Z/p\Z)$. This view point is useful for generalisations.
\begin{remark}
We note that in \cite{LeeSun19}, the slightly larger outcome space $\{a/q\mid 0<a<q\leq Q, (a,q)=1\}$ is considered (following Mazur and Rubin), that is, without the condition that $N|q$. In fact, equidistribution on this outcome space does {\it not} hold in the generality above. One has to exclude some bad primes $p$ (see Remark \ref{strictlyspeaking} below). Our methods can also deal with this larger outcome space, by considering the Fourier expansion of Eisenstein series at different cusps, as is done in \cite{PeRi} or \cite{petru}. The outcome space $\Omega_{Q,N}$ above is, however, very natural from the cohomological perspective and for simplicity we will restrict to this case.  
\end{remark}
\subsection{Distribution of cohomology classes} More generally, let  $\SO(n+1,1)$ be the special orthogonal group with signature $(n+1,1)$, which we identity with the group of isometries of the $(n+1)$-dimensional upper half space $\H^{n+1}.$ Now, for a co-finite subgroup with cusps $\Gamma< \SO(n+1,1)$, we will study the distribution of unitary characters of $\Gamma$ or, equivalently, cohomology classes in $H^1(\Gamma, \R/\Z)$. These cohomology groups have been studied in many contexts (\cite{Sarnak90},\cite[Chap. 7]{ElGrMe98}) and especially the case $n=2$ is very appealing as it corresponds to Kleinian groups due to the exceptional isomorphism $\SO(3,1)\cong \SL_2(\C)$.
\subsubsection{Results with arithmetic ordering}
Let $\Gamma\subset \SO(n+1,1)$ be as above and assume that the associated symmetric space $\Gamma\backslash \H^{n+1}$ has a cusp at $\infty$. Let $\Gamma_\infty' \subset \Gamma$ be the parabolic subgroup fixing the cusp at $\infty$. Note that since $\Gamma$ is discrete, there exists a lattice $\Lambda < \R^n$ such that $\Gamma_{\infty}'$ is exactly the group of motions corresponding to translations by $\Lambda$. We will study the distribution of unitary characters trivial on $\Gamma_\infty'$ or, equivalently, elements of the cohomology group $H^1_{\Gi}(\Gamma, \R/\Z)$.

Our distribution results are with respect to a natural arithmetic ordering on $\gag$ which generalises the ordering in the definition of $\Omega_{Q,N}$ above. To define this, we use the {\it Vahlen model} $\mathrm{SV}_{n-1}$  for the group of isometries of $\H^{n+1}$
 consisting of $2 \times 2$ matrices over a specific Clifford algebra, introduced in \cite{Ah86} (see Section \ref{vahlen} below for a detailed construction). This model provides a natural generalisation to $n>2$ of the familiar models $\mathrm{SV}_0=\SL_2(\R)$ and $\mathrm{SV}_1=\SL_2(\C)$. 
 We define the following outcome space:
\begin{equation}\label{outcomespace}
    T_{\Gamma}(X) = \left \{\gamma\in \gag  \mid  0 < |c_\gamma| < X \right \},
\end{equation}
where $\gamma=\begin{psmallmatrix} a_\gamma & b_\gamma \\ c_\gamma& d_\gamma \end{psmallmatrix}\in \mathrm{SV}_{n-1}$ in the Vahlen group model and $|\!\cdot\!|$ denotes the norm on the relevant Clifford algebra. This generalizes the outcome space (\ref{Omega}) above and the ones considered for $n=1$ in \cite{PeRi}, \cite{No19} and for $n=2$ in \cite{petru}.

Now let $\omega_1,\ldots, \omega_d$ be elements of $H^1_{\Gi}(\Gamma, \R/\Z)$ in {\it general position}, meaning that for any $(n_1,\ldots, n_d)\in \Z^d$, we have 
$$ n_1 \omega_1+\ldots +n_d \omega_d=0\in H^1_{\Gi}(\Gamma, \R/\Z) \Leftrightarrow \left(n_i\omega_i=0\in H^1_{\Gi}(\Gamma, \R/\Z),\forall i=1,\ldots, d\right).$$
As an example one can pick $\omega_1, \dots, \omega_d$ to be a $\F_p$-basis for $H^1_{\Gi}(\Gamma, \Z/p\Z)$. We notice that the image of any $\omega\in H^1(\Gamma, \R/\Z)$ is either dense in $\R/\Z$ or finite (recall that $\omega$ defines an additive character $\Gamma\rightarrow \R/\Z$). In the first case we put $J_\omega=\R/\Z$ and in the latter case we put $J_\omega=\Z/m\Z$, where $m$ is the cardinality of the image of $\omega$. We equip $\R/\Z$ and $\Z/m\Z$ with the obvious choices of probability measures, Lebesque and uniform respectively. Finally associated to $\gamma\in \gag$, we define the invariant $\gamma \infty \in (\R^n\cup \{ \infty \})/\Lambda $ using the action of $\SO(n+1,1)$ on the boundary of $\H^{n+1}$, see Section \ref{cofinite groups} for more details. Then we have the following distribution result.


\begin{thm} \label{mainthmHn} Let $\omega_1,\ldots, \omega_d\in H^1_{\Gi}(\Gamma, \R/\Z)$ be in general position. The random variables $\gamma\mapsto (\omega_1(\gamma),\ldots, \omega_d(\gamma), \gamma \infty)$ defined on the outcome spaces $T_{\Gamma}(X)$ are asymptotically uniformly distributed on $\prod_{i=1}^{d} J_{\omega_i} \times (\R^n/\Lambda)$ as $X\rightarrow \infty$. More precisely, for any fixed (continuity) subsets $A_i\subset J_{\omega_i}$ and $B \subset \R^n / \Lambda$, we have
\begin{equation*}
\frac{ \# \left \{ \gamma  \in T_{\Gamma}(X) \mid (\omega_1(\gamma),\ldots, \omega_d(\gamma)) \in \prod_{i=1}^d A_i, \gamma \infty \in B \right\}}{\# T_{\Gamma}(X)} =  \prod_{i=1}^d \frac{|A_i|}{|J_{\omega_i}|} \cdot  \frac{|B|}{\vol(\R^n/\Lambda)}+ o(1)
\end{equation*}
 as $X \to \infty$. 
\end{thm}
\begin{remark}
The Vahlen model has been used before to study automorphic forms on $\H^{n+1}$, for example by Elstrodt, Grunewald, and Mennicke \cite{vahlen2} to prove a generalisation of the Selberg Conjecture regarding the first non-zero eigenvalue of the Laplacian and by  S\"{o}dergren \cite{So12} for proving equidistribution of horospheres on $\H^{n+1}$.
\end{remark}

\begin{remark}
 Notice that the number of choices of cohomology classes in $H^1_{\Gamma_\infty'}(\Gamma, \R/\Z)$ in general position is infinite unless $\Gamma/\langle [\Gamma,\Gamma], \Gamma'_\infty\rangle$ is torsion. See Appendix \ref{cohomology} for a survey of results on the size of $H^1_{\Gamma_\infty'}(\Gamma, \R/\Z)$. 
\end{remark}

The structure of the paper is as follows. In Section \ref{sketch} we give a sketch proof of Theorem \ref{mainthm1} in its simplest form and underline the main ideas, including the use of the analytic properties of the Eisenstein series. In Section \ref{special case}, we prove Theorem \ref{special case thm} using that Hecke characters define unitary characters of congruence subgroups, which in turn are connected to Eisenstein congruences. In Section \ref{geometry}, we introduce some geometric and arithmetic properties of $\H^{n+1}$, including the Vahlen model to study cofinite subgroups of isometries. In Section \ref{eisenstein}, we develop the analytic properties of the $n$-dimensional twisited Eisenstein series and the spectral properties of the twisted Laplacian. In Section \ref{main results}, we use the tools developed in the previous sections to prove our main results. In Appendix \ref{cohomology}, we give a detailed literature survey on the structure of the cohomology groups, to which our results apply.

\section*{Acknowledgements}
We would like to thank Paul Gunnells, Yiannis Petridis, Morten Risager, Djordje Mili{\'c}evi{\'c} for useful discussions and feedback and to Ian Kiming for his help with \cite{Mazur77}. The first author was supported by the Engineering and Physical Sciences Research Council [EP/L015234/1], the EPSRC Centre for Doctoral Training in Geometry and Number Theory (The London School of Geometry and Number Theory), University College London.

\section{Idea of proof}
\label{sketch}
We will sketch the proof of Theorem \ref{mainthm1} in the simplest case, which is the one dealt with in \cite{LeeSun19}, where we consider only one cusp form for $\H^2$ and no restrictions on the location of $r=a/q$ in $\R / \Z$.  
Our method is automorphic in nature and relies on the theory of Eisenstein series. It can be seen as a discrete version of the method introduced by Petridis and Risager in \cite{PeRi2} for studying the distribution of modular symbols. They consider the perturbation of the family of characters $\chi^\eps$ as $\eps\rightarrow 0$, whereas we consider the discrete family $\chi
^m$ for $m\in\Z$. 

Let $f\in \mathcal{S}_2(\Gamma_0(N))$ be a Hecke eigenform of weight 2 and level $N$ and let $\mathfrak{m}_{f}^\pm:\Gamma_0(N) \rightarrow \Z$ be the associated normalised modular symbols defined above. Recall that this defines a non-trivial additive character. We would like to show that the values of $\mathfrak{m}_{f}^\pm$ on the set $\Omega_{Q,N}=\{a/q\mid 0<a<q\leq Q, (a,q)=1, N|q\}$ equidistribute mod $p$ as $Q\rightarrow \infty$. 

To do this we introduce for any $l\in (\Z/p\Z)^\times$ the unitary character $\chi_l:\Gamma_0(N)\rightarrow \C^\times$ defined by
$$ \chi_l(\gamma):= e^{2\pi i \mathfrak{m}_{f}^\pm(\gamma)l /p},\quad \gamma\in \Gamma_0(N).$$
By Weyl's Criterion \cite[page 487]{IwKo} in order to conclude equidistribution, it suffices to detect cancelation in the Weyl sums; that is to prove for all $l\in (\Z/p\Z)^\times$ that
$$  \sum_{a/q\in \Omega_{Q,N}} \chi_l(a/q)=o(Q^2), $$
as $\Q\rightarrow \infty$, where $\chi_l(a/q):= \chi_l\left(\gamma\right)$ with $\gamma\in \Gamma_0(N)$ such that $\gamma \infty=a/q$. 

Now, the key observation is that the generating series for these Weyl sums appears very naturally as the constant term of an appropriate Eisenstein series. The cancelation in the Weyl sums is now a simple analytic consequence of the analytic properties of the corresponding Eisenstein series. To be precise; associated to $\chi_l$ we have the following twisted Eisenstein series:
$$  E(z,s,\chi_l)= \sum_{\gamma\in \Gamma_\infty\backslash \Gamma_0(N)} \overline{\chi_l}(\gamma) \Im (\gamma z)^s,$$  
where $\Gamma_\infty= \left< \begin{psmallmatrix}1 & 1\\ 0 & 1\end{psmallmatrix}\right>$. This Eisenstein series defines a holomorphic function for $\Re s>1$ and by the work of Selberg \cite[Chap. 39]{Selberg1} admits meromorphic continuation to the entire complex plane with a pole at $s=1$ if and only if $\chi_l$ is trivial. Note that in general the character $\chi_l$ might not come from an adelic one, but Selberg's theory applies equally well.

Now a standard calculation using Poisson summation shows that the constant term of the Fourier expansion of $E(z,s,\chi_l)$ is given by 
$$  y^s+\frac{\pi^{1/2} y^{1-s}\Gamma(s-1/2)}{\Gamma(s)}L_{l}(s), $$
with
$$L_{l}(s):=\sum_{c>0,N\mid c} \left(\sum_{0<d<c, (c,d)=1} \overline{\chi_{l}}\left(\begin{psmallmatrix} a &b \\ c & d\end{psmallmatrix} \right)  \right) c^{-2s},$$
where $\begin{psmallmatrix} a &b \\ c & d\end{psmallmatrix}$ is a(-ny) matrix in $ \Gamma_0(N)$ with lower entries $c,d$. We observe that $L_{l}(s)$ is exactly the generating series for the Weyl sums above, as promised.  

Now from the meromorphic continuation of the Eisenstein series itself, we also get meromorphic continuation of the generating series $L_{l}(s)$, and since $\chi_l$ is non-trivial we conclude that $L_l(s)$ is analytic for $\Re s>1-\delta$ for some $\delta>0$. Thus we get the wanted cancelation in Weyl sums using the standard machinery from complex analysis if we can get bounds on vertical lines of $L_l(s)$. It turns of that such bounds follow from the general bound for scattering matrices also due to Selberg, and thus we are done. 

This shows how to deduce equidistribution of modular symbols using Eisenstein series. The proof for classes in the first cohomology of quotients of higher dimensional hyperbolic spaces uses the same idea, although some parts of the argument require some more technical work. In order to obtain equidistribution results when restricting the cusps to a specific interval $I\subset \R/\Z$, we will have to use all the Fourier coefficients of the Eisenstein series as is done in \cite{PeRi}. 
\section{Some special cases of the conjecture of Mazur and Rubin}\label{special case}
In this section we will consider certain special cases of the conjectures of Mazur and Rubin (and the generalization to $\H^3$), which we can resolve {\it without} taking an extra average. These special cases correspond to the fact that Hecke characters define unitary characters of congruence subgroups, which in turn are connected to Eisenstein congruences as studied intensively by Mazur in \cite[Section 9]{Mazur77} and \cite{Mazur79}. 

First of all we will define the relevant cohomology classes and introduce the Hecke operators in this context. Recall that for a discrete, cofinite subgroup $\Gamma\subset \SL_2(k)$ with $k=\R$ or $\C$ and an element $\alpha\in \tilde{\Gamma}$ of the commensurator of $\Gamma$, we have a decomposition
$$ \Gamma \alpha \Gamma= \bigsqcup_{i=1}^d \Gamma \alpha_i   $$
for some $\alpha_1,\ldots, \alpha_d\in \tilde{\Gamma}$. 
Using this we define the {\it Hecke operator} $T_\alpha$ acting on the cohomology group $H^1(\Gamma, X)$ with $X$ a trivial  $\Gamma$-module as:
\begin{align} \label{heckedef}(T_\alpha \omega)(\gamma):= \sum_{i=1}^d  \omega(\gamma_i), \end{align}
where $\alpha_i\gamma= \gamma_i \alpha_{\sigma(i)}$ with $\gamma_i\in\Gamma$ and $\sigma$ some permutation of $\{1,\ldots d\}$. 

We will consider the case of congruence subgroups 
$$\Gamma_0(\mathfrak{f})=\{ \gamma\in \SL_2(\mathcal{O}_K)\mid \gamma \equiv \begin{psmallmatrix} \ast& \ast\\ 0 & \ast  \end{psmallmatrix} \modulo \mathfrak{f}\},$$
where $K$ is equal to $\Q$ or an imaginary quadratic extension thereof and $\mathfrak{f}$ is a non-trivial ideal of $\mathcal{O}_K$. In this case we have $\widetilde{\Gamma_0}(\mathfrak{f})=\GL_2(K)$ and the parabolic subgroup fixing $\infty$ is $\Gamma_\infty'= \begin{psmallmatrix} \pm 1 & \mathcal{O}_K\\ 0& \pm 1\end{psmallmatrix}$. We say that a Hecke operator is {\it good} if it is of the form $T_\alpha$, where $\alpha= \begin{psmallmatrix} a & 0\\ 0 & 1   \end{psmallmatrix}$ with $\gcd(\mathfrak{f}, (a))=1$.

\begin{prop}\label{eisensteincongruence}
Let  $ m$ be an odd integer diving $ |\left(\mathcal{O}_K/\mathfrak{f}\right)^\times|$. Then there exists a class $\omega\in H^1_{\Gamma_\infty'}(\Gamma_0(\mathfrak{f}), \Z/m\Z )$, which is an eigenvector for all good Hecke operators and such that for all $a\in \Z/m\Z$ and $c_0\in \mathfrak{f}$, it satisfies
\begin{align}\label{onthenose} \frac{\#\{ \gamma\in \Gamma_\infty\backslash \Gamma_0(\mathfrak{f})/\Gamma_\infty \mid c_\gamma=c_0, \omega(\gamma)=a\}}{\#\{ \gamma\in \Gamma_\infty\backslash \Gamma_0(\mathfrak{f})/\Gamma_\infty \mid c_\gamma=c_0\}}=\frac{1}{m},  \end{align}
where $c_\gamma$ denotes the lower left entry of $\gamma$. 
\end{prop}
\begin{proof}
Let $\chi: \left(\mathcal{O}_K/\mathfrak{f}\right)^\times \rightarrow \C^\times $ be a unitary Hecke character of order $m$. Then we define a character of $\Gamma_0(\mathfrak{f})$ by 
\begin{equation}\label{omegaX}\begin{psmallmatrix} a& b\\ c& d \end{psmallmatrix} \mapsto \chi(d).\end{equation}
This character is clearly trivial on $\Gamma_\infty'$ since the order $m$ of $\chi$ is odd, and thus (\ref{omegaX}) defines an element $\omega_\chi\in H^1_{\Gamma_\infty'}(\Gamma_0(\mathfrak{f}), \Z/m\Z)$. For $T_\alpha$ a good Hecke operator, it is easy to check that $\alpha_i$ (with notation as in (\ref{heckedef})) can all be chosen of the form $\begin{psmallmatrix} \ast& \ast\\ 0 & \ast  \end{psmallmatrix} $ with determinant equal to the determinant of $\alpha$ (thus the diagonal entries are coprime to $\mathfrak{f}$).  Combining this with $\gamma_i=\alpha_i \gamma \alpha_{\sigma(i)}^{-1}$, one easily sees that
$$ \left(T_\alpha(\omega_\chi)\right)(\gamma)=d \omega_\chi(\gamma) ,  $$
where $d=|\Gamma_0(\mathfrak{f})\backslash\Gamma_0(\mathfrak{f}) \alpha \Gamma_0(\mathfrak{f})|$. This shows that $\omega_\chi$ is an Hecke eigenclass with eigenvalue $d$, as wanted. 

Finally, recall the basic fact that a set of representatives of  $\Gamma_\infty \backslash \Gamma_0(\mathfrak{f})/\Gamma_\infty$ is given by 
$$\{\begin{psmallmatrix} \ast & \ast\\ c & d   \end{psmallmatrix}\mid c\in \mathfrak{f}, d\in (\mathcal{O}_K/(c))^\times\}.  $$
From this, the equidistribution statement (\ref{onthenose}) follows directly. 
\end{proof}
It is a natural question to ask how the cohomology classes constructed above are related to the modular symbols defined in (\ref{oddevenMS}). To tackle this we need to understand so-called \emph{Eisenstein congruences}, which have been studied intensively by Mazur \cite{Mazur77}. We will now introduce some required terminology and refer to \cite{Mazur77} for a detailed account: We say that a pair of primes $(N,p)$ with $N,p\geq 5$ and $p|N-1$ is \emph{admissible} if the local ring $\mathbb{T}_\mathfrak{P}$ has rank $1$ over $\Z_p$ where $\mathbb{T}$ is the Hecke algebra of level $N$ and $\mathfrak{P}\subset \mathbb{T}$ is the Eisenstein prime corresponding to $p$. In classical terms $(N,p)$ being admissible means that there is a unique cuspidal Hecke eigenform of level $N$ which is congruent to the Eisenstein series of weight $2$ (i.e. $f\in \mathcal{S}_2(\Gamma_0(N))$ s.t. the Hecke eigenvalues satisfy $\lambda_f(l)\equiv l+1\modulo p$ for primes $l\neq N$ and $Uf=-f$ where $U$ is the Hecke operators at $N$). By a computation of Merel \cite{Merel96}$, (N,p)$ is admissible exactly if 
$$\prod_{k=1}^{p-1}((k(N-1)/p)!)^k,$$
is a $p$-power in $(\Z/N\Z)^\times$. Note that all pairs of primes $(N,p)$ with $N<250$ are admissible unless $N=31,103,127,131,181,199,211$ (see the remark on \cite[p. 141]{Mazur77}). In the admissible case we have the following strengthening of Proposition \ref{eisensteincongruence} (see \cite[Chapter II, Proposition 18.8]{Mazur77} for a very related result).
\begin{thm}
For an admissible pair of primes $(N,p)$  with $N,p\geq 5$ and $p|N-1$, there exists a Hecke eigenform $f\in \mathcal{S}_2(\Gamma_0(N))$ of weight 2 and level $N$ such that the values of $\mathfrak{m}^+_{f}$ (defined as in (\ref{oddevenMS})) on $\{\frac{a}{q}\mid (a,q)=1, 0<a<q\}$ equidistribute exactly modulo $p$ for $q\equiv 0\modulo p$.
\end{thm}
\begin{proof}
Let $\chi$ be a Dirichlet character mod $N$ of order $p|N-1$. Then by Proposition \ref{eisensteincongruence} we have an associated cohomology class $\omega_\chi\in H^1_{\Gamma_\infty}(\Gamma_0(N), \Z/p\Z)$ which equidistributes as above and such that $T_l \omega_\chi=(l+1)\omega_\chi$ for all primes $l\neq N$, where $T_l$ is the Hecke operator corresponding to the matrix $\begin{psmallmatrix} l & 0 \\ 0& 1\end{psmallmatrix}$. Furthermore, $\omega_\chi$ satisfies $U\omega_\chi=-\omega_\chi$, where $U$ is the Hecke operator at the bad prime $N$ given by conjugation by $\begin{psmallmatrix} 0 & 1\\N& 0 \end{psmallmatrix}$. Also $\omega_\chi$ is  trivial on the stabilizer $\langle \begin{psmallmatrix}\pm1 & 0\\ 1 & \pm1\end{psmallmatrix}\rangle$ of the cusp $0$ (using that the order of $\chi$ is odd) and thus $\omega_\chi$ defines a parabolic cohomology class. By a {\lq\lq}mod $p$ version{\rq\rq} of  Eichler--Shimura isomorphism (as in \cite[(3.5)]{LeeSun19}), we see that the associations $f\mapsto \mathfrak{m}_{f}^\pm$ for Hecke eigenforms $f$ induce a Hecke-equivariant isomorphism 
\begin{align} \label{modpES}H_P^1(\Gamma_0(N), \Z/p\Z)\cong  \mathcal{S}_2(\Gamma_0(N))_{\F_p}\oplus\mathcal{S}_2(\Gamma_0(N))_{\F_p}\end{align} 
where $\mathcal{S}_2(\Gamma_0(N))_{\F_p}$ denotes the space of cusp forms of weight 2 and level $N$ with coefficients in $\F_p$ (which we will just think of as the formal $\F_p$-vector space generated by Hecke eigenforms of weight 2 and level $N$). Here we use that there is no $p$-torsion in $\Gamma_0(N)$. By the assumption that $(N,p)$ is admissible we conclude that there exists a Hecke eigenform $f\in \mathcal{S}_2(\Gamma_0(N))$ such that $\omega_\chi$ is a linear combination of $\mathfrak{m}_f^\pm$.




Finally, we recall that $H_P^1(\Gamma_0(N), \Z/p\Z)$ can be diagonalized by the involution $\iota$  given by conjugation with $\begin{psmallmatrix} -1 & 0\\0& 1 \end{psmallmatrix}$ (here we need $p>2$), which follows from e.g. \cite[Sec. 1]{MaRu}. We see directly that the eigenvalue of  $\omega_\chi$ under the action of $\iota$ is $+1$. Thus we conclude that $\mathfrak{m}^+_{f}=m\cdot \omega_\chi$ for some $m\in (\Z/p\Z)^\times$. This gives the wanted. 
\end{proof}

This settles the conjecture of Mazur and Rubin in these very special cases, whereas in general the conjecture seems out of reach without the extra average both with the automorphic and the dynamical approach. 
\begin{remark}\label{strictlyspeaking}
Strictly speaking the conjecture of Mazur and Rubin \cite{MaRu} is only formulated for primes $p$ and cusp forms corresponding to elliptic curves $E$ where the residual representation of $E$ mod $p$ is surjective and $p$ is an ordinary and good prime of $E$. This is not the case in the example considered above, but the above seems like the natural generalization of the conjecture to this case. 
\end{remark} 
\begin{remark}
The assumption that $N$ is prime is essential for the results of \cite{Mazur77} to apply. For composite level (and for imaginary quadratic fields) the situation becomes much more complicated as \emph{multiplicity one} might fail (see e.g. \cite{WakeWang}). 
\end{remark}

 \section{Geometry of $\H^{n+1}$}
 \label{geometry}

We introduce the upper half-space (Poincar\'e) model $\H^{n+1}$ for the $(n+1)$-dimensional hyperbolic space. We briefly describe some geometric and arithmetic properties of the space $\Gamma \backslash \H^{n+1}$, where $\Gamma$ is a cofinite discrete subgroup of isometries. We make use of a specific model for the group of isometries given in terms of a certain Clifford algebra. Our main references for this section are \cite{Ah86}, \cite{vahlen1} and \cite{vahlen2}.  


\subsection{Clifford algebra}
We will now describe the {\it upper-half space model} $\mathbb{H}^{n+1}$ for hyperbolic $(n+1)$-space. Let $q:\R^n \to \R$ a quadratic non-degenerate form and $\mathcal{C}(q)$ the associated {\it Clifford algebra}, i.e. the free $\R$-algebra on $\{e_1, \dots, e_n \}$ modulo the relations
\begin{equation*}
    e_i^2= q(e_i), \quad e_i e_j =-e_j e_i, \quad \text{where} \quad   i,j=1, \cdots, n, \ i\neq j ,
\end{equation*}
where $e_1, \dots, e_n$ is a $q$-orthonormal basis for $\R^n$. We denote by $\mathcal{E}_n$ the set of all subsets of $\{ 1, \dots, n\}$. Then for $M= \{i_1, \dots, i_k \} \in \mathcal{E}_n$ with $i_1 < \cdots < i_k$, we define
\begin{equation*}
    e_M := e_{i_1} \cdot \dots \cdot e_{i_k}, \quad e_{\emptyset} := 1 \in \mathcal{C}(q).
\end{equation*}
Then one can check that $\{ e_M \mid M \in \mathcal{E}_n\}$ is a $\R$-basis for $\mathcal{C}(q)$.

We have two linear involutions on $\mathcal{C}(q)$ given by
\begin{align*}
    \overline{e_M} := (-1)^{|M|(|M|+1)/2}e_M, \quad e_M^* : = (-1)^{|M|(|M|-1)/2}e_M, \quad \text{where } M \in \mathcal{E}_n.
\end{align*}
They satisfy
\begin{equation*}
    \overline{vw} = \overline{w} \ \overline{v}, \quad (v w)^* = w^* v^*, \quad \text{for all } v,w \in \mathcal{C}(q).
\end{equation*}

From now on we assume that $q=-I_n$, the negative definite unit form, and $e_1,\ldots, e_n$ the standard basis. In this case we write $\mathcal{C}_n$ for $\mathcal{C}(q)$. We denote by $V_n \subset \mathcal{C}_n$ the vector space spanned by $\{1, e_1, \dots, e_n \}$. It is easy to see that $V_0 \cong \R$ and $V_1 \cong \C$ as $\R$-algebras.  

$V_{n}$ is equipped with the inner product
\begin{equation*}
    \inprod{v}{w}=\frac{1}{2}(v \overline{w}+ \overline{v}w ) .
\end{equation*}
We note that this coincides with the standard Euclidean inner product if we identify $V_n$ with $\R^{n+1}$ using the basis $\{1, e_1, \dots, e_n \}$.

For $x= \sum_{M \in \mathcal{E}_n} \lambda_M e_M \in \mathcal{C}_n$, we define the norm
\begin{equation}
    \label{norm}
    |x|:= \left ( \sum_{M \in \mathcal{E}_n} \lambda_M^2\right )^{1/2}.
\end{equation}
We note that for $x \in V_n$, we have $|x|^2=\inprod{x}{x}$.
Now, if $\Lambda < V_n$ is a lattice, we define the dual lattice as
\begin{equation*}
    \Lambda^{\circ}:= \{w \in V_n \ | \ \inprod{v}{w} \in \Z \text{ for all } v \in \Lambda \}.
\end{equation*}

We now define the following model of hyperbolic $(n+1)$-space:
\begin{equation*}
    \label{upper space}
    \H^{n+1}:= \{x_0 + x_1 e_1 + \dots + x_n e_n \ | \ x_0, x_1, \dots, x_{n-1} \in \R, x_n > 0 \} \ .
\end{equation*}
We have the maps $x: \H^{n+1} \to V_{n-1}$ and $y:\H^{n+1} \to (0, \infty)$ given by
\begin{equation*}
    x(P):= x_0 +  x_1 e_1 + \dots + x_{n-1} e_{n-1} , \quad y(P):= x_n,
\end{equation*}
where $P=x_0 + x_1 e_1 + \dots + x_n e_n \in \H^{n+1}$. We can think of $x(P)$ as an element of $\R^n$ via the above. Then from \eqref{norm} we see that
\begin{equation*}
    |P|^2=|x(P)|^2 + |y(P)| ^2.
\end{equation*}

We equip $\H^{n+1}$ with the hyperbolic metric coming from the line element:
\begin{align}
    \label{linemetric}
    ds^2=\frac{dx_0^2+dx_1^2+\dots+dx_n^2}{x_n^2} \ ,
\end{align}
which makes $\H^{n+1}$ a Riemannian manifold with constant negative curvature $-1$. The volume element is given by
\begin{equation*}
    \label{volumemetric}
    dv=\frac{dx_0 dx_1 \dots x_n}{x_n^{n+1}} .
\end{equation*}

The {\it hyperbolic Laplace--Beltrami operator} is  given by
\begin{equation}
    \label{laplaceoperator}
    \Delta = x_n^2 \left ( \pdv[2]{}{x_0} + \pdv[2]{}{x_1} + \dots + \pdv[2]{}{x_n} \right ) - (n-1) x_n \pdv{x_n}
\end{equation}
in this model.

\subsection{Vahlen group}
\label{vahlen}
We will use the above upper-half space model to describe the group of (oriented) isometries $\mathrm{Isom}^+(\H^{n+1})$ in a way that is convenient for our purposes. We let $T_n \subset \mathcal{C}_n$ be the multiplicative subgroup generated by $V_n \setminus \{ 0 \}$. As in \cite[p. 219]{Ah86} or \cite[p. 648]{vahlen2}, we define the {\it Vahlen group} $\mathrm{SV}_n$ to be
\begin{equation}
    \label{SV_n}
    \mathrm{SV}_n := \left \{ 
\begin{tabular}{ c|c } 
\multirow{3}{9em}{$\begin{pmatrix} a & b \\ c & d \end{pmatrix} \in M_2 (\mathcal{C}_n)$} & (i) $a,b,c,d \in T_n \cup \{0\}$ \\ 
 &  (ii) $\overline{a}b,\overline{c}d \in V_{n}$ \\ 
 &  (iii) $ad^*-bc^*=1$ \\
\end{tabular}
    \right \}.
\end{equation}
We can easily check that $\mathrm{SV}_0=\mathrm{SL}_2(\R)$ and $\mathrm{SV}_1=\mathrm{SL}_2(\C)$ as $\R$-algebras. Then it is a non-trivial fact that $\mathrm{SV}_n$ is a group under matrix multiplication with inverse
\begin{equation}
    \label{inverse matrix}
    \begin{pmatrix} a & b \\ c & d\end{pmatrix}^{-1}= \begin{pmatrix} d^* & -b^* \\ -c^* & a^* & \end{pmatrix} \ .
\end{equation}

We can now define the action of $\mathrm{SV}_{n-1}$ on $\H^{n+1}$, which resembles the actions of $\SL_2(\R)$ and $\SL_2(\C)$ on $\H^2$ and $\H^3$, respectively, as can be seen from the following result.

\begin{thm}[\cite{vahlen2}, Theorem 1.3]
Let $\gamma = \begin{psmallmatrix} a &b \\ c & d\end{psmallmatrix} \in \mathrm{SV}_{n-1}$ and $P \in \H^{n+1}$. Then $cP+d \in T_n$ and we define
\begin{equation}
    \label{action}
    \gamma P := (aP+b) (cP+d)^{-1} \in \H^{n+1}.
\end{equation}
The map $P \mapsto \gamma P$ is an orientation preserving isometry of $\H^{n+1}$. Moreover, all orientation preserving isometries are obtained in this way and we have the induced isomorphism $\mathrm{SV}_{n-1} / \{ I , -I\} \cong \mathrm{Isom}^+(\H^{n+1})$.
\end{thm}
 What is convenient about this description of $ \mathrm{Isom}^+(\H^{n+1})$ is that one gets very familiar expressions for the coordinate-projections of the image under the action of $\gamma \in \mathrm{SV}_{n-1}$ on $P=(x,y) \in \H^{n+1}$.
\begin{lemma}[\cite{vahlen2}, page 648]
Let $\gamma = \begin{pmatrix} a &b \\ c & d\end{pmatrix} \in \mathrm{SV}_{n-1}$ and $P=x + y e_n \in \H^{n+1}$. Then
\begin{equation}
    \label{y of action}
    x(\gamma P) = \frac{(ax+b)(\overline{cx+d}) + a \overline{c} y^2 }{|cx+d|^2 + |c|^2 y^2} \quad \text{and} \quad y(\gamma P) = \frac{y}{|cx+d|^2 + |c|^2 y^2}.
\end{equation}
\end{lemma}

\begin{remark}
Our model for the hyperbolic $(n+1)$-space is consistent with other descriptions from the literature. For example, one can consider the Klein model $\K^{n+1}$ on which isometries are described by $\mathrm{SO}(n+1,1)$. Then there exists an bijection $\Phi:\H^{n+1} \to \K^{n+1}$ and an isomorphism $\Psi:\mathrm{SV}_{n-1}/\{ \pm I\} \xrightarrow{\sim} \mathrm{SO}^0(n+1,1)$ which commutes with the respective actions, i.e. $\Phi(\gamma \cdot P)= \Psi(\gamma) \Phi(P),$
for all $\gamma \in \mathrm{SV}_{n-1}$ and $P \in \H^{n+1}$. Here $\mathrm{SO}^0(n+1,1)$ is the component of the identity element in $\mathrm{SO}(n+1,1)$. We refer to \cite[Section 5]{vahlen1} for detailed descriptions of different models of the hyperbolic space.
\end{remark}

\subsection{Hyperbolic quotients}
\label{cofinite groups}
Let $\Gamma<\mathrm{SV}_{n-1}$ be a discrete subgroup of motions such that the surface $\Gamma \backslash \H^{n+1}$ has finite hyperbolic volume. 
We say that $\mathfrak{a} \in \R^n \cup \{\infty\}$ is a cusp for $\Gamma$ if it is fixed by a parabolic element in $\Gamma$. There exists a scaling matrix $\sigma_{\mathfrak{a}}\in  \mathrm{SV}_{n-1}$ such that $\sigma_{\a} \infty = \a$. We let $\Gamma_{\a} : = \{ \gamma \in \Gamma \mid \gamma \a = \a \}$ be the stabilizer of $\a$ in $\Gamma$. We define 
\begin{equation*}
    \Gamma_{\a}':= \Gamma_{\a} \cap \sigma_{\a}  \left \{  \begin{psmallmatrix} 1 & b \\ 0 & 1\end{psmallmatrix} \in \mathrm{SV}_{n-1} \right \} \sigma_{\a}^{-1} \ .
\end{equation*}
We note that $\Gamma_{\a}'$ consists of the parabolic elements in $\Gamma_{\a}$ together with the identity.

There exists a lattice $\Lambda_{\a} \leq \R^n$ such that 
\begin{equation*}
    \sa^{-1} \Gamma_{\a}' \sa =  \left \{ \begin{pmatrix} 1 & \lambda \\ 0 & 1 \end{pmatrix} \mid \lambda \in \Lambda_{\a} \right \} \ .
\end{equation*}
We let $\mathcal{P}_{\a}$ be a fundamental parallelogram for $\Lambda_{\a}$ with Euclidean area $\vol(\Lambda_{\a})$. 

We define the {\it dual lattice} of $\Lambda_{\a}^{\circ}$ as follows:
\begin{equation}
    \label{dual lattice}
    \Lambda_{\a}^{\circ}:= \{ \mu \in \R^n \mid \inprod{\mu}{\lambda} \in \Z \text{ for all } \lambda \in \Lambda_{\a} \} \ ,
\end{equation}
where $\inprod{\cdot}{\cdot}$ is the usual scalar product on $\R^n$.

For a cusp $\a$ and $Y>0$, we define the {\it cuspidal sector}
\begin{equation*}
    \mathcal{F}_{\a}(Y):=  \sigma_{\a} \{ (x,y) \mid x \in \mathcal{P}_{\a}   , y > Y\} \ .
\end{equation*}
Then for $Y$ large enough, there exists a fundamental domain $\mathcal{F}$ for $\Gamma \backslash \H^{n+1}$ and inequivalent cusps $\a_1, \cdots, \a_h \in \R^n \cup \{\infty\}$ such that we can write $\mathcal{F}$ as the disjoint union
\begin{equation}\label{funddomain}
    \mathcal{F}=\mathcal{F}_0 \sqcup \mathcal{F}_{\a_1}(Y) \sqcup \cdots \sqcup \mathcal{F}_{\a_h}(Y) \ ,
\end{equation}
where $\mathcal{F}_0$ is a compact set, see \cite[p. 8]{So12} or \cite[p. 5]{Sarnak90}.

For notational convenience, from now on we will focus only on the cusp at $\infty$. We drop the subscript by denoting $\Lambda:=\Lambda_{\infty}$, $\mathcal{P}:=\mathcal{P}_{\infty}$ etc. Our theory can be generalised to take all cusps into account.

We will now define our outcome space (\ref{outcomespace}) in precise terms. First we note that all elements in $\gag$ share the same lower left entry. Thus it makes sense to define
\begin{equation*}
    T_{\Gamma}(X):= \left \{ \begin{pmatrix} * & * \\ c & * \end{pmatrix} \in \gag \mid 0 < |c| \leq X \right \}, 
\end{equation*}
where $|c|$ denotes the Clifford norm (\ref{norm}). This is the natural generalisation of the outcome space considered by Petridis--Risager in \cite[p. 1002]{PeRi}. In (\ref{size of T(X)}) below, we provide an asymptotic formula for the size of $T_{\Gamma}(X)$. We put
\begin{equation}\label{C(gamma)}
    C(\Gamma):= \left \{ c \in T_n \mid \exists a,b,d\in T_n: \begin{psmallmatrix} a &b \\ c& d \end{psmallmatrix} \in \Gamma\right\}.
\end{equation}

If $\gamma = \big (\begin{smallmatrix} a & b \\ c & d \end{smallmatrix}\big) \in \Gamma$ then from the definition of the action \eqref{action}, we see that $\gamma \infty = a c^{-1}$, where $\gamma \infty$ is defined as the limit of $\gamma P$ as $P$ tends to the cusp at $\infty$. Also, from \cite[Lemma 1.4]{vahlen2}, we know that $a c^{-1} \in V_{n-1}$. 

We observe that $\gamma \infty$ is well-defined on double cosets in $\gag$ up to translations by the lattice $\Lambda$. Therefore we see that the map
\begin{align*}
    \gag &\to \R^n / \Lambda \cup \{ \infty \} \\
    \gamma &\mapsto \gamma \infty
\end{align*}
is well-defined using the identification of $V_{n-1}$ with $\R^n$ as above. A simple consequence of our main theorems is that $\gamma \infty$ become equidistributed on $\R^n / \Lambda$ as we vary along $ \gamma \in T_{\Gamma}(X)$ as $X \to \infty$.

\section{Twisted Eisenstein series for $\H^{n+1}$}
\label{eisenstein}
Let $\Gamma <\mathrm{SV}_{n-1}$, $\Gamma_\infty'$ and $\Lambda$ be as in the previous section. We now fix $\chi$ a unitary character of $\Gamma$ which is trivial on $\Gamma_\infty'$. From this we define the twisted Eisenstein series
\begin{equation}
    \label{eisenstein series}
    E(P,s,\chi)=\sum_{\Gamma_{\infty}' \backslash \Gamma} \overline{\chi(\gamma)} y(\gamma P)^s.
\end{equation}

It is absolutely convergent for $\Re(s)>n$ and satisfies
\begin{align*}
  E(\gamma P,s, \chi) &= \chi(\gamma) E(P, s,\chi), \\ 
  \Delta E(P, s,\chi) &= s(n-s) E(P,s,\chi) .
\end{align*}

We see that $E(P, s,\chi)$ is invariant under the action by the lattice $\Lambda$ and hence it has a Fourier expansion. It is well-known that the constant term in the Fourier expansion has the form $y^s+\phi(s,\chi) y^{n-s}$, where $\phi(s, \chi)$ is called the {\it scattering matrix}. Its basic properties are well-known,  see \cite[Ch. 6]{SarnakNotes}.

For $\mu, \nu \in \Lambda^{\circ}$ and $c \in C(\Gamma)$, we define the generalised Kloosterman sum as in \cite[Section 4]{vahlen2} using the Vahlen model:
\begin{align}
    \label{kloos}
    S(\mu, \nu , c, \chi) &:= \sum_{ \big (\begin{smallmatrix} a & b \\ c & d \end{smallmatrix}\big) \in \Gamma_{\infty}' \backslash \Gamma / \Gamma_{\infty}'} \overline{\chi} \left ( \begin{pmatrix}  a & b \\ c & d\end{pmatrix} \right ) e \left ( \inprod{a c^{-1}}{ \mu} + \inprod{d c^{-1}}{\nu}\right )\\
    &=\sum_{\substack{\gamma \in \gag \\ c_{\gamma}=c}} \overline{\chi(\gamma)}e(\inprod{\gamma \infty}{\mu} + \inprod{(\gamma^{-1} \infty)^*}{\nu}),
\end{align}
where $c_\gamma$ is the lower-left entry of $\gamma$ in the Vahlen model. We now calculate the Fourier expansion of the Eisenstein series using the techniques developed in \cite[p. 111--113]{ElGrMe98} and \cite[p. 676--678]{vahlen2}. We obtain 

\begin{align}
    \nonumber
    E(P,s, \chi)= [\Gamma_{\infty}:\Gi] y^s &+ y^{n-s}\frac{\pi^{n/2} \Gamma\left (s-\frac{n}{2} \right )}{{\vol(\Lambda)} \Gamma(s)} L(s, \chi) \\
   \label{fourier expansion} &+ \frac{2 \pi^s y^{n/2}}{{\vol(\Lambda)} \Gamma(s)}\sum_{\mu \in \Lambda^{\circ} \setminus \{0 \}} L(s, \mu, \chi) |\mu|^{s-n/2}K_{s-n/2}(2 \pi |\mu| y),
\end{align}
where
\begin{equation}
    \label{definition of L0}
    L(s, \chi): =   \sum_{\gamma \in T_{\Gamma}} \frac{\overline{\chi}(\gamma)}{|c_\gamma|^{2s}} = \sum_{c \in C(\Gamma)}\frac{S(0,0, c, \chi)}{|c|^{2s}}\ ,
\end{equation}
and for $\mu\neq 0$,
\begin{equation}
    \label{Lmu}
    L(s, \chi,\mu): = \sum_{\gamma \in T_{\Gamma}} \overline{\chi}(\gamma) \frac{e(\inprod{d_\gamma c_\gamma^{-1}}{\mu})}{|c_\gamma|^{2s}}= \sum_{c \in C(\Gamma)} \frac{S(0, \mu, c, \chi)}{|c|^{2s}}.
\end{equation}
For $\chi=1$ the trivial character, we just denote $L(s, \mu):=L(s, \mu, 1)$. 
We note that the explicit Fourier expansion we obtain in (\ref{fourier expansion}) is closely related to \cite[Thm. 9.1]{vahlen2}.

At other cusps $\mathfrak{a}\neq \infty$ of $\Gamma$, we will also need some information about the Fourier expansion. For this let $P^{\mathfrak{a}}=(x^{\mathfrak{a}},y^{\mathfrak{a}} )=\sigma_\mathfrak{a}^{-1}P$ denote the coordinates at $\mathfrak{a}$. Then the Fourier expansion at $\mathfrak{a}$ is given by \cite[Ch. 6, Prop. 1.42]{SarnakNotes}:
$$  E(P^{\mathfrak{a}},s, \chi)= \phi_\mathfrak{a}(s) (y^\mathfrak{a})^{n-s}+\sum_{\mu\in  \Lambda_{\a}^{\circ} \setminus \{0\}}\phi_\mathfrak{a}(s,\mu) (y^\mathfrak{a})^{n-s}K_{s-n/2}(2\pi n |\mu| y^{\mathfrak{a}}) e(\langle x^\mathfrak{a}, \mu\rangle),  $$
where $\phi_\mathfrak{a}(s,\mu)$ are the Fourier coefficients, which decay rapidly in $|\mu|$ (for $s$ fixed). 
In particular we observe that $E(P,s, \chi)$ is square integrable when restricted to $\mathcal{F}_\mathfrak{a}(Y)$ for $\mathfrak{a}\neq \infty$ (for $Y$ sufficiently large as in (\ref{funddomain})). 
\begin{remark}\label{51}
By inverting $\gamma$ in the definition of $L(s,  \chi, \mu)$, we observe that
\begin{align}
\label{def2}    L(s, \chi,\mu) = \sum_{\gamma \in T_{\Gamma}} \overline{\chi}(\gamma) \frac{e(\inprod{(\gamma^{-1} \infty)^*}{\mu})}{|c_\gamma|^{2s}} 
= \sum_{\gamma \in T_{\Gamma}} \chi(\gamma) \frac{e(\inprod{\gamma \infty}{\mu})}{|c_\gamma|^{2s}}.
\end{align}
\end{remark}

\subsection{Short discussion on spectral properties}
We say that a (measurable) function $f: \H^{n+1}\rightarrow \C$ is {\it $\chi$-automorphic} if it satisfies 
\begin{equation*}
 f(\gamma P)= \chi(\gamma) f(P) \ ,
\end{equation*}
for $P\in \H^{n+1}$ and $\gamma\in \Gamma$.

Denote by $\Ltwo$ the space of square integrable $\chi$-automorphic functions with respect to the hyperbolic metric. For $f,g \in \Ltwo$, we note that $f \overline{g}$ is $\Gamma$-invariant. Hence we can define the inner product
\begin{equation*}
    \inprod{f}{g}:= \int_{\mathcal{F}} f \overline{g} \ d v \ .
\end{equation*}

We let $\mathcal{D}(\chi) \subset \Ltwo$ be the subspace consisting of all $C^2$-functions such that $\Delta f \in \Ltwo$. Then one can see that $-\Delta: \mathcal{D}(\chi) \to \Ltwo $ is a symmetric and nonnegative operator, its spectrum consists of discrete and continuous parts with finitely many discrete points in the interval $[0,n^2/4)$. Let
\begin{equation*}
    0 \leq \lambda_0(\chi) \leq \lambda_1(\chi) \leq \cdots \leq \lambda_k(\chi)<n^2/4
\end{equation*}
be the eigenvalues in the interval $[0,n^2/4)$ (see
 \cite{Sarnak90} and \cite[Ch. 6]{SarnakNotes}). The Eisenstein series $E(z,s,\chi)$ admits meromorphic continuation to $s \in \C$ and satisfies the functional equation
\begin{equation*}
    E(P, n-s, \chi)=\phi(n-s, \chi) E(P,s,\chi) \ ,
\end{equation*}
where $\phi(s,\chi)$ is the scattering matrix. Moreover, $E(P,s,\chi)$ has poles where $\phi(s, \chi)$ has poles and viceversa. There are finitely many poles in the region $\Re(s)>n/2$, all of them simple and on the real line. If $n/2 < \sigma_0 \leq n$ is a pole of $E(P,s, \chi)$, denote by $u_{\sigma_0}$ its residue at $\sigma_0$. Then
\begin{equation*}
    u_{\sigma_0} \in \Ltwo \quad \text{and} \quad \Delta u_{\sigma_0}+\sigma_0(n-\sigma_0) u_{\sigma_0} = 0 \ .
\end{equation*}
For $0 \leq j \leq k$, let $s_j(\chi) \in (n/2, n]$ be such that $s_j(\chi)(n-s_j(\chi))=\lambda_j(\chi)$. We denote by
\begin{equation*}
    \label{spectrum}
    \Omega(\chi):= \{ s_{0}(\chi), \dots, s_{k}(\chi)\}.
\end{equation*}
Then the poles of $E(P,s,\chi)$ in $\Re s>n/2$ form a subset of $\Omega(\chi)$ (exactly the non-cuspidal part of the discrete spectrum). Moreover, we can see from \cite[Ch 6, p. 37]{SarnakNotes} that for $\chi$ trivial, we have
\begin{equation}
    \label{residue value}
    \Res_{s=n}E(P,s)=\frac{[\Gamma_{\infty}: \Gi]\vol(\Lambda)} {\mathrm{vol}(\GH)} \ .
\end{equation}

\subsection{Key lemmas}

In this section we will prove certain key analytic lemmas that we will need in the proofs of our theorems. First of all we will show that we can only have $\lambda_0(\chi)=0$ when $\chi$ is trivial. Secondly we obtain meromorphic continuation of the Fourier coefficients of the twisted Eisenstein series, which will serve as generating series for our distribution problems. Finally, we will prove a bound on vertical lines for these generating series.

The most conceptual way to see the first claim above is probably to use Green's identity
\begin{equation*}
    \int_{\mathcal{F}}(-\Delta u) u  dv = \int_{\mathcal{F}}\nabla u . \nabla u  \ dv + \int_{\partial \mathcal{F}} u (\nabla u . \bf{n}) d S.
\end{equation*}
If we have $\Delta u=0$, then the first integral is 0. The third integral should vanish since contributions from {\lq\lq}opposing sides{\rq\rq} in the boundary of the fundamental domain should cancel each other. This would force the second integral to be 0, which means $u$ is constant. This argument works in principle, but for example in \cite[Theorem 4.1.7]{ElGrMe98} they spend several pages making it rigorous in the three dimensional case. Instead we will give an argument using the Fourier expansion and the mean value theorem for harmonic functions.

\begin{lemma} \label{triviallemma}
We have that $\lambda_0(\chi)=0$ if and only if $\chi$ is trivial.
\end{lemma}

\begin{proof}
Suppose $\lambda_0(\chi)=0$ and let $u$ be a corresponding eigenvector, i.e. $u \in \Ltwo$ and $\Delta u = 0$. Then we can consider the Fourier expansion of $u$ at a cusp $\mathfrak{a}$ of $\Gamma$. We know from \cite[Ch. 6, p.10]{SarnakNotes} that the Fourier expansion of $u$ takes the form 
$$ c_{1,\mathfrak{a}}+c_{2,\mathfrak{a}} (y^\mathfrak{a})^n+\sum_{\mu\in \Lambda_{\a}^\circ\backslash\{0\}} a_{u, \mathfrak{a}}(\mu) (y^\mathfrak{a})^{n/2}K_{n/2}(2\pi n |\mu| y) e(\langle x,\mu\rangle).  $$
From the rapid decay of the $K$-Bessel function we see that if $c_{2,\mathfrak{a}}\neq 0$, then $u$ behaves like $(y^\mathfrak{a})^n$ close enough to $\mathfrak{a}$ and thus $\int_{F_\mathfrak{a}(Y) }|u(x,y)|^2 dxdy$ is divergent contradicting the fact that $u$ is square integrable. Thus $c_{2,\mathfrak{a}}= 0$ and we conclude again using the rapid decay of the $K$-Bessel functions that $u$ is bounded on $F_\mathfrak{a}(Y)$. Since $\mathfrak{a}$ was an arbitrary cusp we conclude that $u$ is bounded on all of $\mathcal{F}$. Thus since $\chi$ is unitary, we conclude that $u$ is bounded on all of $\H^{n+1}$. Now it follows from the {\it Mean Value Theorem for Harmonic Functions on} $\H^{n+1}$ that $u$ is constant. By definition, $u(\gamma P)=\chi(\gamma) u(P)$, for all $\gamma \in \Gamma$ and $P\in \H^{n+1}$. Thus we conclude that $\chi$ is the trivial character.

Therefore, if $\chi$ is trivial the unique eigenfunction of eigenvalue 0 is the constant one, and for $\chi$ non-trivial there are no eigenfunctions of eigenvalue 0. This finishes the proof.\end{proof}

We now obtain meromorphic continuation of the Fourier coefficients of the Eisenstein series and crucial information about the location of the poles.

\begin{prop}
\label{properties}
The Dirichlet series $L(s, \mu, \chi)$ admits meromorphic continuation to the entire complex plane. The possible poles in the half-plane $\Re s>n/2$ are contained in $\Omega(\chi)$. Furthermore, there is a pole at $s=n$ exactly if $\chi$ is trivial and $\mu=0$. In this case the residue is equal to
$$ \frac{[\Gamma_{\infty}: \Gi]\Gamma(n) \vol(\Lambda)^2} {\pi^{n/2} \Gamma\left (\frac{n}{2} \right )\vol(\GH)}  .$$
\end{prop}
\begin{proof}
From \eqref{fourier expansion}, we know that for $\mu \in \Lambda^{\circ}\backslash\{0\}$
\begin{equation*}
    L(s, \mu, \chi)=\frac{\Gamma(s)}{2 \pi^s y^{n/2} |\mu|^{s-n/2} K_{s-n/2}(2 \pi |\mu| y)} \int_{\mathcal{P}}E((x,y),s, \chi)e(-\inprod{x}{\mu}) dx,
\end{equation*}
and 
\begin{equation*}
    L(s,\chi)=\frac{y^{s-n}\Gamma(s)} {\pi^{n/2} \Gamma\left (s-\frac{n}{2} \right )} \left(\int_{\mathcal{P}}E((x,y),s, \chi) dx-[\Gamma_{\infty}: \Gi]y^s\right),
\end{equation*}
where $\mathcal{P}$ is a fundamental parallelogram for $\Lambda$. Now for $y>0$ fixed , the Bessel function $K_s(y)$ defines an analytic function in $s$, which is non-zero for some $y$ large enough. Similarly the Gamma function defines a meromorphic function. Thus we get the meromorphic continuation of $L(s,\mu,\chi)$ from that of the Eisenstein series. We also note that in the half-plane $\Re s>n/2$, $L(s, \mu, \chi)$ has possible poles only where $E(P,s, \chi)$ has poles, i.e. the poles are contained in $ \Omega(\chi)$. By Lemma \ref{triviallemma}, we see that $L(s,\mu,\chi)$ is regular at $s=n$ unless $\chi$ is trivial. 

If $\chi$ is trivial, we see that $L(s,\mu)$ with $\mu\neq 0$ is regular at $s=n$, since the pole of the Eisenstein series is constant. For $\mu=0$ the residue is given by
$$  \Res_{s=n}L(s,0)=\frac{\Gamma(n)} {\pi^{n/2} \Gamma\left (\frac{n}{2} \right )} \int_{\mathcal{P}}\frac{[\Gamma_{\infty}: \Gi]{\vol(\Lambda)}}{\vol(\GH)} dx= \frac{[\Gamma_{\infty}: \Gi]\Gamma(n){\vol(\Lambda)}^2} {\pi^{n/2} \Gamma\left (\frac{n}{2} \right )\vol(\GH)} , $$
as wanted.
\end{proof}

In order to obtain bounds on vertical lines for our generating series, we will use ideas due to Colin de Verdi\`{e}re \cite{CodeVe83}, which employs the analytic properties of resolvent operators. Alternatively, one could use Poincar\'{e} series for $\mu\neq 0$ and Maa{\ss}--Selberg for $\mu=0$ as is done in \cite{PeRi} and \cite{petru}. In the end the two methods are essentially equivalent. 

Let $h: \R^+ \to \R^+$ be a smooth function which is equal to $[\Gamma_{\infty}: \Gi]$  for $y>Y+1$ and 0 for $y<Y$, where $Y$ is as in (\ref{funddomain}). Then for $\Re (s)>n/2$ we define a $\chi$-automorphic function on $\H^{n+1}$ by $P\mapsto h(y)y^s$ for $P\in \mathcal{F}$ and extended periodically (twisted accordingly by $\chi$). Then from the above mentioned results on the Fourier expansions of the Eisenstein series at the different cusps, we see that 
$$  g(P,s,\chi):=E(P,s,\chi)-h(y)y^s\in L^2(\Gamma\backslash \H^{n+1},\chi), $$
which satisfies for $z\in \mathcal{F}$
$$ (\Delta-s(n-s))g(P,s,\chi)=-(\Delta-s(n-s))h(y)y^s= h^{\prime\prime}(y)y^{s+2}+(2s-n+1)h^\prime (y)y^{s+1}. $$
We observe that the right hand side above is compactly supported with $L^2$-norm bounded by $O(|s|+1)$ for $n/2+\eps<\Re s<n+2$. Now we put 
$$ H(P,s,\chi):= R(s,\chi) (h^{\prime\prime}(y)y^{s+2}+(2s-n+1)h^\prime (y)y^{s+1})\in L^2(\Gamma\backslash \H^{n+1},\chi), $$
where $R(s,\chi)=(\Delta-s(n-s))^{-1}$ denotes the resolvent operator associated to $\Delta$. By a general bound for the operator norm of resolvent operators \cite[Lemma A.4]{Iw}, we conclude that
$$  |\!|H(\cdot,s,\chi)|\!|_{L^2}\ll_\eps 1,   $$
when $s$ is bounded at least $\eps$ away from $\Omega(\chi)$. 
We can now write
\begin{equation}\label{Collin}  E(P,s,\chi)=H(P,s,\chi)+h(y)y^s, P\in \mathcal{F} \end{equation}
where we have good control on the $L^2$-norm of $H(P,s,\chi)$.
We will use this to obtain bounds on vertical lines for the Fourier coefficients of $E(P,s,\chi)$, mimicking \cite[Section 4.4]{No19}.

\begin{prop}
\label{vertical lines}
Let $\mu\in \Lambda^\circ$. Then we have 
$$L(s,\mu, \chi)\ll_{\eps,\mu} (|s|+1)^{n/2},$$ for $n/2+\eps<\Re s<n+2$ and $s$ bounded at least $\eps$ away from $\Omega(\chi)$. 
\end{prop} 
\begin{proof}
We have
\begin{align}\label{Fourierbound} L(s, \mu,\chi) =  \int_{\mathcal{P}} f_s(y,\mu)E((x,y),s,\chi) e(-\langle x, \mu \rangle) dx-\mathbf{1}_{\mu=0} [\Gamma_{\infty}: \Gi] y^s f_s(y,\mu),  \end{align}
where $\mathbf{1}_{\mu=0}$ is $1$ if $\mu=0$ and $0$ otherwise and
$$f_s(y,\mu)=\begin{cases} \Gamma(s)\left(2\pi^s y^{n/2}|\mu|^{s-n/2}K_{s-n/2}(2\pi n |\mu|y)\right)^{-1}, & \mu \neq 0,\\  \Gamma(s)\left(y^{n-s}\pi^{n/2}\Gamma(s-n/2)\right)^{-1}, & \mu=0. \end{cases}$$
The idea is now to bound the right hand side of (\ref{Fourierbound}) using (\ref{Collin}). In order to bring the information we have about $H(P, s,\chi)$ into play, we need to make an extra integration over $y$. So let $Y$ be a fixed quantity such that $\{(x,y)\mid x\in \mathcal{P}, y>Y\}\subset \mathcal{F}$, then we see that 
\begin{align*}
& \int_{Y}^{Y+1} \int_{\mathcal{P}} f_s(y,\mu)E((x,y),s,\chi) e(-\langle\mu, x \rangle) dx dy\\
&= \int_{Y}^{Y+1} \int_{\mathcal{P}} f_s(y,\mu)H((x,y),s,\chi) e(-\langle\mu, x \rangle) dx dy\\
&+\int_{Y}^{Y+1} \int_{\mathcal{P}} f_s(y,\mu) h(y)y^s e(-\langle\mu, x \rangle) dx dy\end{align*}

Now we observe that by Cauchy--Schwarz we have
\begin{align*}
&\int_{Y}^{Y+1} \int_{\mathcal{P}} f_s(y,\mu) H((x,y),s,\chi) e(-\langle\mu, x \rangle) dx dy\\
&\leq \left(\int_{Y}^{Y+1} \int_{\mathcal{P}} |H((x,y),s,\chi)|^2  dx dy\right)^{1/2}\left(\int_{Y
}^{Y+1} \int_{\mathcal{P}} |f_s(y,\mu)|^2  dx dy\right)^{1/2} \\
&\ll |\! |H(\cdot, s,\chi)|\! |_{L^2}\left(\int_{Y
}^{Y+1} |f_s(y,\mu)|^2  dy\right)^{1/2},
\end{align*}
where we use that $\{(x,y)\mid x\in \mathcal{P}, y>Y\}\subset \mathcal{F}$. To finish the proof we need an upper bound for $f_s(y,\mu)$. 

For $\mu=0$ we get by Stirling's approximation the upper bound 
$$f_s(y,0)\ll_{\eps} y^{n-\sigma}(|s|+1)^{n/2},$$
for $s=\sigma+it$ with $n/2+\eps <\sigma<n+2$. 

For $\mu\neq 0$, we use the Fourier expansion for the $K$-Bessel function (coming from combining \cite[(B.32)]{Iw} and \cite[(B.34)]{Iw}) to obtain a good approximation. By applying Stirling's approximation, this gives for $s=\sigma+it$ with $t\gg 1$
\begin{align*}
     K_{s-n/2}(2\pi |\mu| y)&=\frac{\pi^{1/2}t^{\sigma-n/2-1/2} e^{\pi t/2}\left(\frac{t}{e}\right)^{it}}{2\sqrt{2} \sin(\pi(s-n/2))} \left(\pi  |\mu|y \right)^{-s+n/2} (1+O_{\mu,y}(t^{-1}))\\
    &\gg_{\mu,y} e^{-\pi t/2}t^{\sigma-n/2-1/2}, 
\end{align*}
where the implied constants depend continuously on $y$. From this we conclude that when $y\in (Y,Y+1)$, we have
$$ f_s(y,\mu)\ll_{\mu} (1+|s|)^{n/2}.   $$
Inserting this and using the bound $|\! |H(\cdot, s, \chi)|\! |_{L^2}\ll_\eps 1$, we conclude that 
\begin{align*}
   L(s,\mu,\chi)\ll_{\eps, \mu} (|s|+1)^{n/2},
\end{align*}
for $s$ bounded $\eps$ away from $\Omega(\chi)$, as wanted.
\end{proof}

Using this we deduce the following asymptotic expression using a standard complex analysis argument. See \cite[p. 20--21]{petru} or \cite[Appendix A]{No19} for fully detailed proofs in similar settings.
\begin{prop}
\label{main}
  Let $\chi$ be a unitary character of $\Gamma$ trivial on $\Gamma_\infty'$ and $\mu \in \Lambda^{\circ}$. Then there exists a constant $\nu(\chi)>0$ such that 
\begin{equation*}
    \sum_{\gamma \in T_{\Gamma}(X)} \chi(\gamma)e\left(\inprod{\gamma \infty}{\mu}\right)=  \frac{X^{2s_0(\chi)}}{s_0(\chi)}\left ( \Res_{s=s_0(\chi)}L(s, \chi, \mu) + O_{\chi,\mu}(X^{-\nu(\chi)})\right ).
\end{equation*}
\end{prop}

\begin{proof}
 Let $\phi_U: \R \to \R$ be a family of smooth non-increasing functions with
    \begin{equation}
        \label{phi lemma}
        \phi_U(t)= \begin{cases} 1 \quad \text{if } t\leq 1-1/U, \\ 0 \quad \text{if } t  \geq 1+ 1/U
        \end{cases}
    \end{equation}
    and $\phi_U^{(j)}(t) = O(U^j)$ as $U \to \infty$. For $\Re (s)>0$, we consider the Mellin transform
    \begin{equation}
        R_U(s) = \int_{0}^{\infty} \phi_{U}(t) t^{s} \frac{dt}{t} \ .
    \end{equation}
    Now we use Mellin inversion  and \eqref{def2} to obtain
    \begin{equation*}
        \sum_{\gamma \in T_{\Gamma}}  \chi(\gamma) e\left(\inprod{\gamma \infty}{\mu}\right)\  \phi_U \lr{\frac{|c|^2}{X}} = \frac{1}{2 \pi i} \int_{\Re(s)=n+1} L(s, \chi,\mu ) X^{s} R_U(s) ds.
    \end{equation*}
    We move the line of integration to $\Re(s)=h(\chi)$ for some $h(\chi)>n/2$ such that $s_1(\chi)< h(\chi) < s_0(\chi)$. We use the fact that we have polynomial growth on vertical lines for $L(s, \chi, \mu)$ guaranteed by Lemma \ref{vertical lines} and that $L(s, \chi, \mu)$ has only a possible pole at $s_0(\chi)$ in the region $\Re(s)>h(\chi)$. We conclude that

    \begin{equation*}
        \label{almost there}
         \sum_{\gamma \in T_{\Gamma}}  \chi( \gamma) e\left(\inprod{\gamma \infty}{\mu}\right)  \phi_U \lr{\frac{|c|^2}{X}} = \frac{X^{s_0(\chi)}}{s_0(\chi)}\lr{\Res_{s=s_0(\chi)} L(s, \chi, \mu)+ O_{\chi,\mu, U}(X^{-\nu(\chi)})} \ ,
    \end{equation*}
    for some $\nu(\chi)>0$. Also, with the appropriate choice of $U$, one can show that
    \begin{align*}
       \sum_{\gamma \in T_{\Gamma}}  \chi( \gamma) e\left(\inprod{\gamma \infty}{\mu}\right)  \phi_U \lr{\frac{|c|^2}{X}} = \sum_{\gamma \in T_{\Gamma}(\sqrt{X})}  \chi( \gamma ) e\left(\inprod{\gamma \infty}{\mu}\right) + O_{\chi,\mu}(X^{n-a(\chi)}),
   \end{align*}
   for some $a(\chi)>0$.
   The conclusion follows. 
\end{proof}
\begin{remark} As a consequence of Proposition \ref{main} and Proposition \ref{properties}, we conclude that for all unitary characters $\chi$ as above, there exist $\nu(\chi)>0$ such that
\begin{equation*}
    \sum_{\gamma \in T_{\Gamma}(X)} \chi(\gamma)e\left(\inprod{\gamma \infty}{\mu}\right)= \mathbf{1}_{\chi, \mu}   \frac{\vol(\Lambda)^2 \Gamma(n)}{n \pi^{n/2} \vol(\GH)  \Gamma(n/2)} X^{2n} +O_\chi(X^{2n-\nu(\chi)}),
\end{equation*}
where $\mathbf{1}_{\chi, \mu}$ is $1$ if $\mu=0$ and $\chi$ is trivial and $0$ otherwise. In particular, we conclude
\begin{align}
    \label{size of T(X)}
    \#T_\Gamma(X)\sim \frac{\vol(\Lambda)^2 \Gamma(n)}{n \pi^{n/2} \vol(\GH)  \Gamma(n/2)}X^{2n},
\end{align}
as $X\rightarrow \infty$.
\end{remark}

\section{Proof of main results}
\label{main results}


In this section we will use the analytic properties of twisted Eisenstein series proved in the previous section to proof our main results.  

We recall the setup from the introduction. Consider the cohomology group $H^1_{\Gamma_\infty'}(\Gamma, \R/\Z)$ (see Appendix \ref{cohomology} for details), which can be identified with the set of unitary characters of $\Gamma$ trivial on $\Gi$. 
\begin{defi}We say that $\omega_1,\ldots, \omega_d\in H^1_{\Gamma_\infty'}(\Gamma, \R/\Z)$ are in {\bf general position} if for any $(l_1,\ldots, l_d)\in \Z^d$, we have 
$$ n_1 \omega_1+\ldots +n_d \omega_d=0\in H^1_{\Gi}(\Gamma, \R/\Z) \Leftrightarrow \left(n_i\omega_i=0\in H^1_{\Gi}(\Gamma, \R/\Z),\forall i=1,\ldots, d\right).$$

\end{defi}

As an example one can pick $\omega_1, \dots, \omega_d$ to be a $\F_p$-basis for $H^1_{\Gamma_\infty'}(\Gamma, \Z/p\Z)$, where we consider $\Z/p\Z\subset \R/\Z$ via $\Z/p\Z\ni a\mapsto a/p$. 

The image of any $\omega\in H^1(\Gamma, \R/\Z)$ is an additive subgroup of $\R/\Z$ and thus is either dense in $\R/\Z$ or finite. In the first case we put $J_{\omega}=\R/\Z$ and in the latter case we put $J_\omega=\Z/m\Z$ where $m$ is the cardinality of the image of $\omega$. That is, $J_{\omega}$ is the closure of the image of $\omega$. We equip $\R/\Z$ and $\Z/m\Z$ with respectively the Lebesque measure and the uniform probability measure.

\begin{proof}[Proof of Theorem \ref{mainthmHn}]Let $\omega_1,\ldots, \omega_d\in H^1_{\Gamma_\infty'}(\Gamma_0(N), \R/\Z)$ be in general position. Then for any tuple $\underline{l}=(l_1,\ldots, l_{d})\in \Z^d$ such that $l_i \omega_i\neq 0\in H^1_{\Gamma_\infty'}(\Gamma_0(N), \R/\Z)$ for all $i=1,\ldots, d$,
we get a non-trivial element of $H^1_{\Gamma_\infty'}(\Gamma, \R/\Z)$ defined by
$$\omega_{\underline{l}}:=l_1 \omega_1+\ldots+l_d\omega_d.$$
Now we consider the associated non-trivial unitary character $\chi_{\underline{l}}: \Gamma\rightarrow \C^\times$ given by
$$ \chi_{\underline{l}}(\gamma):= e\left( \omega_{\underline{l}}(\gamma) \right), $$ 
where $e(x)=e^{2\pi i x}$. Observe that this is indeed well-defined and that we get an induced map $\chi_{\underline{l}}: \gag \rightarrow \C^\times$ since $\omega_{\underline{l}}$ is trivial on $\Gamma_\infty'$. 

By {\it Weyl's Criterion} \cite[p. 487]{IwKo} in order to conclude equidistribution of the values of 
$$\omega(\gamma) :=(\omega_1(\gamma),\ldots, \omega_d(\gamma),\gamma \infty )$$
inside $\prod_{i=1}^d J_{\omega_i}\times (\R^n/\Lambda)$, we have to show cancelation in the corresponding Weyl sums:
$$  \sum_{\gamma \in T_{\Gamma}(X)} \chi_{\underline{l}}(\gamma)e(\langle \gamma \infty,\mu \rangle), $$
where $\underline{l}\in \Z^d$ and $\mu \in \Lambda^\circ$. We see that it follows from combining Proposition \ref{main} and Remark \ref{51} that we have
$$ \sum_{\gamma \in T_\Gamma(X)} \chi_{\underline{l}}(\gamma)e(\langle  \gamma \infty,\mu\rangle)= o\left(\sum_{\gamma \in T_{\Gamma}(X)} 1\right), $$
as $X\rightarrow \infty$ {\it unless} $\mu=0$ and $\chi_{\underline{l}}$ is trivial. This finishes the proof of Theorem \ref{mainthmHn} using Weyl's Criterion. \end{proof}

\subsection{Distribution of modular symbol mod $p$ and mod $1$} 

Now let us see how Theorem \ref{mainthm1} follows from Theorem \ref{mainthmHn}. 

\begin{proof}[Proof of Theorem \ref{mainthm1}]We restrict to $n=1$ and $\Gamma=\Gamma_0(N)$. By the mod $p$-version of the Eichler--Shimura isomorphism (\ref{modpES}), we see that $\mathfrak{m}_f^\pm$ with $f\in \mathcal{S}_2(\Gamma_0(N))$ give a basis for $H^1_P(\Gamma_0(N),\Z/p\Z)$. Thus it follows that they are in general position and thus we conclude Theorem \ref{mainthm1} after noting that $T_{\Gamma_0(N)}(Q)=\Omega_{Q,N}$.\end{proof} 

A different application is to consider the distribution of un-normalized modular symbols mod $1$. So let $f_1,\ldots, f_d\in \mathcal{S}_2(\Gamma_0(N))$ be a basis of Hecke-normalized new forms and consider the map $\Q\rightarrow (\R/\Z)^{2d+1} $ given by
\begin{align} \Q\ni r\mapsto\mathfrak{m}_{N, \R/\Z}(r)=(\Re \langle r, f_1\rangle, \Im\langle r, f_1\rangle,\ldots,  \Im\langle r, f_d\rangle,r ),\end{align}
as a random variable defined on $\Omega_{Q,N}$ defined as in (\ref{Omega}). 
\begin{cor}
\label{mainthm2}
The random variables $\mathfrak{m}_{N,\R/\Z}$ defined on the outcome spaces $\Omega_{Q,N}$ converge in distribution to the uniform distribution on $(\R/\Z)^{2d+1}$ as $Q\rightarrow \infty$. More preciely, for any fixed product of intervals $\prod_{n=1}^{2d+1} I_n \subset  (\R/\Z)^{2d+1}$, we have
\begin{equation*}
\frac{ \# \left \{ a/q \in \Omega_{Q,N}\cap I_{2d+1}  \mid  (\Re \langle a/q, f_1\rangle, \ldots,  \Im\langle a/q, f_d\rangle ) \in \prod_{n=1}^{2d} I_n \right \}}{\# \Omega_{Q,N}} =  \prod_{n=1}^{2d+1} |I_n|+ o(1)
\end{equation*}
 as $Q \to \infty$. \end{cor}
\begin{proof}From a classical result of Schneider \cite{Schneider37} we know that the periods (or elliptic integrals) $\Omega_{f,\pm}$ appearing in (\ref{oddevenMS}) are transcendental. By the rationality of (\ref{oddevenMS}), this implies that the cohomology class associated to a newform $f$ given by
$$\Gamma_0(N)\ni  \gamma \mapsto \int_{\gamma\infty}^\infty \Re (f(z) dz) $$
takes some irrational value (and similarly for $\Im (f(z)dz)$). Thus by the Eichler--Shimura isomorphism, we conclude that given a  basis $f_1,\ldots , f_d$ of Hecke-normalized newforms, the associated cohomology classes $\Re f_i(z)dz$ and $\Im f_i(z)dz$ are in general position and the images of the associated characters are dense in $\R/\Z$.  

Now Corollary \ref{mainthm2} follows directly from Theorem \ref{mainthmHn}.
\end{proof}

\subsection{Proof of Corollary \ref{dedekind}.}
Now we see how our results can be applied to the residual distribution of Dedekind sums $s(a,q)=\sum_{k=1}^q (\!(k/q)\!)(\!(ak/q)\!)$ where
$$(\!(x)\!)=\begin{cases} x-\lfloor x\rfloor -1/2 ,& x\notin \Z \\ 0 ,& x\in \Z\end{cases}$$
is the {\lq\lq}sawtooth{\rq\rq} function.

\begin{proof}[Proof of Corollary \ref{dedekind}]The results of \cite[Section 5]{Mazur79} shows (after some simple manipulations) that for $N,p$ as in Corollary \ref{dedekind},
$$\Gamma_0(N)\ni \begin{psmallmatrix}a& b\\ Nq& d\end{psmallmatrix}\mapsto s(a, Nq)-s(a,q)-\frac{(N-1)(a+d)}{12 q}$$
defines a non-trivial element $\omega_{N,p}\in H^1_{\Gamma_\infty}(\Gamma_0(N), \Z/p\Z)$ with eigenvalue $-1$ under the involution given by conjugation by $\begin{psmallmatrix}0 & 1\\ N& 0 \end{psmallmatrix}$. Now let $\omega_\chi\in H^1_{\Gamma_\infty}(\Gamma_0(N), \Z/p\Z)$ be the cohomology class associated to a Dirichlet character $\chi$ mod $N$ of order $p$ as in the proof of Theorem \ref{eisensteincongruence}, which we recall has eigenvalue $+1$ under the conjugation action by $\begin{psmallmatrix}0 & 1\\ N& 0 \end{psmallmatrix}$. We observe that $\omega_\chi(\gamma)=a_0'\in \Z/p\Z$ corresponds exactly to $\gamma$ having  upper left entry in some  fixed coset  $a_0H$ of the unique index $p$ subgroup $H$ of $(\Z/N\Z)^\times$. Now Corollary \ref{dedekind} follows directly by applying Theorem \ref{mainthm1} to $\omega_{N,p}$ and $\omega_\chi$.  
\end{proof}

\subsection{On the variance of the residual distribution}
\label{variance section}
A natural question to ask next is how well the values equidistribute in Theorem \ref{mainthmHn}. For simplicity, we will restrict to $\H^2$. So let $\Gamma=\Gamma_0(N)$, $f\in \mathcal{S}_2(\Gamma_0(N))$ be Hecke newform and consider the normalized modular symbols $\mathfrak{m}_{f}^\pm$ as above. In what follows we will suppress $\mathfrak{m}_{f}^\pm$ from the notation. 

We consider for each $X>0$ the random variable $Y_{p,X}$ defined on the outcome space $\Z/p\Z$ (with uniform probability measure) by 
$$  \Z/p\Z \ni a\mapsto \frac{ \# \{ \gamma  \in T_{\Gamma}(X) \mid \mathfrak{m}_{f}^\pm(\gamma) \equiv a \modulo p \}}{\# T_{\Gamma}(X)}. $$
Clearly, we have $\mathbb{E}(Y_{p,X})=\frac{1}{p}$ and Theorem \ref{mainthm1} says that as $X\rightarrow \infty$, the random variable $Y_{p,X}$ converge in distribution to the Dirac measure at $\frac{1}{p}$. We will now calculate the variance, which is a natural measure for the regularity of our distribution problem:
 $$ \mathrm{Var}(Y_{p,X})=\mathbb{E}((Y_{p,X}-\mathbb{E}Y_{p,X})^2)= \frac{1}{p}\sum_{a\in \Z/p\Z} \left(Y_{p,X}(a)-\frac{1}{p}\right)^2. $$
 First of all we observe that for the modular symbols and primes appearing in Theorem \ref{eisensteincongruence}, we have $\mathrm{Var}(Y_{p,X})=0$ for all $X$. On the other hand we can prove using the perturbation theory of the hyperbolic Laplacian, that as $p$ grows, the picture is very different.
 

\begin{thm}
\label{variance}
We have for $p$ large enough
\begin{align}
    \label{result:variance}
    \mathrm{Var}(Y_{p,X}) = c_p X^{4s_p-4}+O_p(X^{4s_p-4-\delta_p}),
\end{align}
for some $s_p, c_p,\delta_p>0$, as $X\rightarrow \infty$. 
As $p\rightarrow \infty$, we have $c_p = 2/p + O(p^{-3})$ and $s_p= 1-c_f p^{-2}+O(p^{-3})$, where $c_f$ is given by \eqref{cf}. 

Furthermore, we can calculate the deviation from the mean for each individual residue class. For $p$ large enough and $a\in \Z/p\Z$, we have: 
\begin{align} \label{bias}\frac{\#\{ \gamma \in T_\Gamma(X)\mid \mathfrak{m}^\pm_f(\gamma)\equiv a\modulo p\}}{\#T_\Gamma(X)}- \frac{1}{p}\sim d_{a,p}X^{2s_p-2},  \end{align}
as $X\rightarrow \infty$, where $d_{a,p}= \frac{2\cos (\frac{2\pi a}{p})}{p}+O(p^{-2})$ as $p\rightarrow \infty$.
\end{thm}
\begin{proof}

For $\eps>0$ we define the character $ \chi_{\eps}:\Gamma_0(N)\to \C$ defined by
$$ \gamma \mapsto e^{2\pi i \mathfrak{m}^\pm_f(\gamma)\eps}.$$ 
Let $ \lambda_{0}(\eps)=s_{0}(\eps)(1-s_{0}(\eps))$ with $ s_{0}(\eps)> 1/2$ be the smallest {\it non-cuspidal} eigenvalue of the hyperbolic Laplacian acting on $\chi_\eps$-automorphic functions (i.e. $s_{0}(\eps)$ is the right-most pole of the twisted Eisenstein series $E(z,s,\chi_\eps)$). Here we put $s_{0}(\eps)=1/2$ if there are no residual eigenvalues.    
From this we define 
$$ s_p:= \max_{a\in (\Z/p\Z)^\times} s_{0}(a/p),  $$
which will turn out to control the variance. Note that $s_p<1$ for all $p$ by Lemma \ref{triviallemma}.

By simple Fourier analysis on $\Z/p\Z$ we have
\begin{align}
   \label{simpleFA} \frac{\# \{ \gamma  \in T_{\Gamma}(X) \mid \mathfrak{m}_{f}^\pm (\gamma)\equiv a \modulo p \}}{\#T_{\Gamma}(X)} = \frac{1}{p}\sum_{b\in \Z/p\Z} \frac{1}{\#T_\Gamma(X)}\sum_{\gamma\in T_{\Gamma}(X)} \chi_{b\!/\!p}(\gamma) e^{-2\pi iab/p } .
\end{align}
By Parseval this implies 
\begin{align}
\label{eq:variance} \mathrm{Var}(Y_{p,X})= \frac{1}{p}\sum_{a\in (\Z/p\Z)^\times}\left|\frac{1}{\#T_{\Gamma}(X)}\sum_{\gamma\in T_{\Gamma}(X)} \chi_{a\!/\!p}(\gamma)\right|^2.
\end{align} 
Now by a contour integration argument as in Proposition \ref{main}, we conclude that if $s_0(a/p)=1/2$ (i.e. there are no non-cuspidal eigevalues in $[0,1/4)$ for the Laplacian acting on $\chi_{a\!/\!p}$-automorphic functions) then  
$$\sum_{\gamma\in T_{\Gamma}(X)} \chi_{a\!/\!p}(\gamma)=O_{\eps}(X^{1+\eps}). $$
On the other hand if $s_0(a/p)>1/2$, then we conclude that
$$\sum_{\gamma\in T_{\Gamma}(X)} \chi_{a\!/\!p}(\gamma)=c_{a,p} X^{2s_{0}(a/p)}(1+O(X^{-\delta_{a,p}})), $$
for some $\delta_{a,p}>0$ depending on the spectral gap between $\lambda_0(a/p)$ and $\lambda_1(a/p)$ and some $c_{a,p}\neq  0$ depending on the constant term of the non-cuspidal eigenfunction corresponding to $\lambda_{0}(a/p)$. Combining this with (\ref{eq:variance}), we deduce the formula (\ref{result:variance}).

We now want to understand the large $p$ behavior. For this we employ perturbation theory of the twisted Laplacian, as developed in \cite[Section 4]{PS91} and \cite{Ep87}. We have that the smallest eigenvalue $\lambda_0(\epsilon)=s_0(\epsilon)(1-s_0(\epsilon))$ of the twisted Laplacian by the character $\chi_{\epsilon}$ is real analytic in $\epsilon$, for $\epsilon$ small enough. Moreover, we know that 
\begin{equation}
    \label{s0eps}
    s_0(\eps)=1-c_f \eps^2+O(\eps^{3}), 
\end{equation}
 as $\eps\rightarrow0$, where
 \begin{equation}
 \label{cf}
     c_f=\frac{8 \pi^2 \| f \|^2}{\vol(\Gamma) \Omega_{f,\pm}^2},
 \end{equation} see \cite[Section 4]{petru} or \cite{PeRi} for more details.
 
 
 Now fix $\epsilon>0$ small enough such that \eqref{s0eps} holds. We want to show that if $\theta \in [\epsilon, 1-\epsilon]$, then $\lambda_0(\theta)$ is bounded away from 0 (and hence $s_0(\theta)$ is bounded away from 1). This follows almost directly from \cite[Proposition 2.1]{Ep87}. Suppose the contradiction, i.e. there exists a sequence $\left \{\theta_j \right \} \subset [\epsilon, 1-\epsilon]$ such that $\lambda_0(\theta_j) \to 0$. By a compactness argument, by passing to a subsequence, we can assume that there exists $\theta^* \in [\epsilon, 1-\epsilon]$ such that $\theta_{j} \to \theta^*$. Denote by $f_j \in L^2(\Gamma \backslash \H, \chi_{\theta_j})$ the corresponding eigenfunctions with eigenvalues $\lambda_0(\theta_j)$. By the continuity statement in \cite[Proposition 2.1]{Ep87}, we conclude that there exists $f^* \in  L^2(\Gamma \backslash \H, \chi_{\theta^*}) $ such that a subsequence of $(f_j)$ is $L^2$-convergent to $f^*$ and $\Delta f^* = 0$. But this means that $f^*$ is constant, and hence $\theta^*$=0, which is a contradiction.
 
By conjugations, we have $s_0(\eps)=s_0(-\eps)$. Using the above and (\ref{s0eps}), we conclude that for $p$ large enough, we have that $s_p=s_0(1/p)=s_0(-1/p)=s_0((p-1)/p)$, which combined with (\ref{s0eps}) gives the wanted.
 
Now, from \eqref{eq:variance}, we note that the main term in the variance is given by the contributions of $a=1$ and $a=p-1$ in the sum.  By (\ref{size of T(X)}) we have  $\#T_\Gamma (X) = (\pi \vol (\Gamma))^{-1} X^2 (1+ O(X^{-\nu}))$, for some $\nu >0$. Furthermore, we know that the eigenfunction (and in particular its constant Fourier coefficient) corresponding to $s_0(\eps)$ varies analytically with $\eps$ (for $\eps$ small enough) and we can deduce that
\begin{equation*}
    \Res_{s=s_0(\epsilon)}L(s, \chi_{\epsilon})= \frac{1}{\pi \vol (\Gamma)}+ O(\epsilon^2), 
\end{equation*}see \cite{petru} for more details. Hence, from \eqref{eq:variance} and Proposition \ref{main}, we deduce that
 \begin{align*}
     c_p=\frac{2}{p} + O(p^{-3}).
 \end{align*}

Finally for $p$ large enough, we see that the main term in (\ref{simpleFA}) comes from $b=0$, and the second main term is given by 
$$\frac{1}{p}( c_{1,p}e^{-2\pi i /p}+c_{p-1,p}e^{2\pi i /p})X^{2 s_0(1/p)-2},$$
which by the above gives (\ref{bias}).
\end{proof} 
We note that the inequality (\ref{bias2}) does indeed follow from (\ref{bias}). 
 \begin{remark}
 We note that it should be straightforward to generalise Theorem \ref{variance} to $\H^n$, as the perturbation theory of the first eigenvalue of the Laplacian has been developed by Epstein \cite{Ep87} for  $\H^n$. 
 \end{remark}

\appendix

\section{On the size of certain cohomology groups}\label{cohomology}
In this paper we study the distribution of certain cohomology classes which can be identified with the unitary characters of cofinite subgroups $\Gamma< \SO(n+1,1)$ (or equivalently $\Gamma< \mathrm{SV}_{n-1}$) with cusps. It is now a natural question to ask how many unitary characters our results actually apply to. This amounts to finding the dimensions of the relevant spaces of unitary characters or equivalently of certain cohomology groups. This last perspective is most useful when comparing it to the existing literature. We will mostly restrict to arithmetic subgroup, which we will define shortly. Then we will define the cohomology groups that are relevant and finally survey what is known about their size.
\subsection{Congruence subgroups}
 
We will now define what we mean by a {\it congruence subgroup}, which most of the results mentioned below applies to. In this case one can obtain quite explicit descriptions of the double coset $\gag$  occuring in Theorem \ref{mainthmHn}.

Let $J \subset \mathcal{C}_n$ be an order stable under the involutions $-$ and $*$. We put $\mathrm{SV}_n(J):=\mathrm{SV}_n \cap M_2(J)$. We also define $V(J):=J \cap V_n$ and $T(J)=J \cap T_n$. For $N \in \N$, we define the {\it principle congruence subgroup}
\begin{equation}
    \label{SV_k(J,n)}
    \mathrm{SV}_n(J;N):= \left \{  \begin{psmallmatrix} a &b \\ c & d \end{psmallmatrix} \in \mathrm{SV}_n(J) \mid a-1, b, c, d-1 \in NJ\right \} .
\end{equation}
A subgroup $\Gamma < \mathrm{SV}_n(J)$ is called a{ \it congruence subgroup} if $\mathrm{SV}_n(J;N)< \Gamma$, for some $N \in \N$. We quote  \cite[Section 4]{vahlen2} to provide an explicit description for representatives of $\gag$ in the case $\Gamma=\mathrm{SV}_n(J;N)$. In this case, $C(\Gamma)= N \cdot  T(J)$ and a set of representatives for $  (\begin{smallmatrix} a & b \\ c & d \end{smallmatrix}\big) \in \gag$ with $c \neq 0$ is given by
\begin{equation*}
    \left \{ \begin{psmallmatrix} a& b\\ c& d \end{psmallmatrix} \in \mathrm{SV}_n(J)\mid c \in N \cdot  T(J), \ (a,d) \in D(c) \right \}
\end{equation*}
where
\begin{align*}
 D(c)&:= \left \{
    \begin{tabular}{ c|c } 
\multirow{2}{3em}{$(a,d)$} & $a \in J/( N \cdot V(J) \cdot c), \ d \in J/(N \cdot c \cdot V(J)), $ \\ 
 &   $a-1,d-1 \in N \cdot J, \ a\overline{c}, \overline{c}d \in N \cdot V(J)$ \\
\end{tabular}\right \}.
\end{align*}

In the more familiar cases $n=1$ and $n=2$, the above reduces to the following.
\begin{itemize}
    \item {\it $n=1$}. Then $\mathrm{SV}_0=\mathrm{SL}_2(\R)$, $J= \Z$ and $\mathrm{SV}_1(J;N)=\Gamma_1(N)$. Representatives in $\Gamma_1(N)_{\infty}' \backslash \Gamma_1(N) / \Gamma_1(N)_{\infty}'$ with $c \neq 0$ are uniquely determined by
    \begin{equation*}
        \left \{ (a,c) \ | \ c>0, \ N \mid c, \ a \in (\Z/cN\Z)^*, \ a \equiv 1 \text{ mod } N\right \}.
    \end{equation*}
    If we consider $\Gamma=\Gamma_0(N)$, then representatives are uniquely determined by
     \begin{equation*}
        \left \{ (a,c)  \ |  \ c>0, \ N \mid c, \ a \in (\Z/c\Z)^* \right \}.
    \end{equation*}
  
\item {\it $n=2$}. Then $\mathrm{SV}_1=\mathrm{SL}_2(\C)$. We take $J=\mathcal{O}_K$, where $\mathcal{O}_K$ is the ring of integers of a quadratic imaginary field $K$. Let $\mathfrak{n} < \mathcal{O}_K$ be an ideal. We consider congruence subgroups of the form
\begin{align*}
    \Gamma_1(\mathfrak{n})&:= \left \{ \begin{psmallmatrix} a &b \\ c&d \end{psmallmatrix} \in \SL_2(\mathcal{O}_K)\mid a-1,b,c, d-1 \in \mathfrak{n} \right \}, \\
    \Gamma_0(\mathfrak{n})&:= \left \{ \begin{psmallmatrix}  a &b \\ c&d \end{psmallmatrix} \in \SL_2(\mathcal{O}_K)\mid c \in \mathfrak{n} \right \}.
\end{align*}
In the case $\Gamma_1(\mathfrak{n})$, representatives are uniquely provided by 
\begin{equation*}
    \left \{ (a,c) \ | \ c \in \mathfrak{n} \setminus \{0 \},  \ a \in (\mathcal{O}_K/(c \cdot \mathfrak{n}))^*, \ a-1 \in \mathfrak{n} \right \},
\end{equation*}
while for $\Gamma_0(\mathfrak{n})$ we have
\begin{equation*}
    \left \{ (a,c) \ | \ c \in \mathfrak{n} \setminus \{0 \},  \ a \in (\mathcal{O}_K/ (c))^* \right \}.
\end{equation*}
\end{itemize}

\begin{remark}
There is also a notion of congruence groups for $\mathrm{SO}(n+1, 1)$. To define them, let $\Gamma$ be the integral automorphisms of an isotropic quadratic form of signature $(n+1,1)$ defined over $\Q$. Then a {\it congruence subgroup of $\Gamma$} is any subgroup containing $\{\gamma\in \Gamma\mid \gamma\equiv I_{n+2}\mod N\}$ for some positive integer $N$, see \cite[p. 7]{Sarnak90}. If $\Gamma < \mathrm{SO}^0(n+1,1)$ is a congruence subgroup, then $\Psi^{-1}(\Gamma)$ is a congruence subgroup in $\mathrm{SV}_{n-1}$. 
However, the converse is not true, there exists a congruence subgroup $\Gamma < \mathrm{SV}_{n-1} $ such that $\Psi(\Gamma)$ is not a congruence subgroup in $\mathrm{SO}^0(n+1,1)$, see \cite[Section 3]{vahlen2} for more details.
\end{remark}
\subsection{The first cohomology group} We refer to \cite[Chapter 8]{Sh75} for a comprehensive account. The {\it first cohomology group} of $\Gamma$ with coefficients in a $\Z[\Gamma]$-module $A$ is defined as the quotient between the corresponding {\it coboundaries} and {\it cocycles};
$$H^1(\Gamma, A):=Z^1(\Gamma, A)/ B^1(\Gamma, A),$$ 
where 
$$Z^1(\Gamma, A):=\{ \omega:\Gamma \rightarrow A\mid \omega(\gamma_1\gamma_2)=\omega(\gamma_1)+\gamma_1.\omega(\gamma_2) ,\forall\gamma_1,\gamma_2\in \Gamma \}  $$
and 
$$B^1(\Gamma, A):=\{ \omega:\Gamma \rightarrow A\mid \exists a\in A: \omega(\gamma)=\gamma.a-a,\forall \gamma\in\Gamma \}.  $$
Furthermore given a subset $P\subset \Gamma$, we will be studying the first $P$-cohomology group of $\Gamma$ with coefficients in $A$ defined by;
$$H_P^1(\Gamma,A):=\{\omega\in H^1(\Gamma,A)\mid \omega(p)\in (p-1)A,\forall p\in P\}.$$ 
We will in particular study the distribution of $P$-cohomology group in the case where $P=\Gi$ is the set of parabolic elements of $\Gamma$ fixing $\infty$ and $A$ is given by the circle $\R/\Z$ equipped with the trivial $\Gamma$-action. In this case $H^1_P(\Gamma,\R/\Z)$ computes exactly the unitary characters of $\Gamma$ trivial on $\Gi$. 

Now we will make some general comments on the structure and size of $H^1_P(\Gamma,\R/\Z)$.

\subsection{On the structure of the cohomology groups}
We recall that for $A$ a trivial $\Gamma$ module we have 
$$ H^1(\Gamma, A) \cong \Hom_\Z(\Gamma/ [\Gamma,\Gamma],A), $$
which is a special case of the {\it Universal Coefficients Theorem} since $H_1(\Gamma, \Z)\cong \Gamma/ [\Gamma,\Gamma]$. From this we see that $H^1(\Gamma, \R/\Z)$ can be identified with the unitary characters of $\Gamma$. It is known \cite[p. 484]{Selberg1} that $\Gamma$ is finitely represented and thus $\Gamma/ [\Gamma,\Gamma]$ is a finitely generated abelian group. From this we see that we 
have a splitting of the cohomology group $H^1(\Gamma, \R/\Z)$ in a free part and a torsion part;
$$  H^1(\Gamma, \R/\Z)\cong H^1_{{\rm free}}(\Gamma, \R/\Z)\oplus H^1_{{\rm tor}}(\Gamma, \R/\Z),$$
where the $\R/\Z$ rank of $ H^1_{{\rm free}}(\Gamma, \R/\Z)$ is the same as the dimension of $H^1(\Gamma, \R)$ and the size of $H^1_{{\rm tor}}(\Gamma, \R/\Z)$ is equal to the size of the torsion in $H_1(\Gamma, \Z)\cong \Gamma/ [\Gamma,\Gamma]$. \\

We have a further Eichler--Shimura splitting of the free part due to Harder \cite{Harder75};
\begin{equation}\label{ESH}H^1(\Gamma, \R)\cong H^1_{{\rm cusp}}(\Gamma, \R)\oplus H^1_{{\rm Eis}}(\Gamma, \R),\end{equation}
where $H^1_{{\rm cusp}}(\Gamma, \R)$ is the cuspidal part corresponding to certain automorphic forms for $\Gamma$ (as we will see shortly) and $H^1_{{\rm Eis}}(\Gamma, \R)$ is the (remaining) Eisenstein part, which can be canonically defined. The cuspidal part $H^1_{{\rm cusp}}(\Gamma, \R)$ can be identified with $H^1_{P}(\Gamma, \R)$ where $P$ is the set of all parabolic elements of $\Gamma$ and furthermore all of the above splittings are compatible with the Hecke action, when $\Gamma$ is arithmetic.

There has been a lot of work recently on the study of the size of respectively $H^1_{{\rm cusp}}(\Gamma, \R)$, $H^1_{{\rm Eis}}(\Gamma, \R)$ and $H^1_{{\rm tor}}(\Gamma, \R/\Z)$, and we will now collect the relevant results for our problem. We observe that the image of $\Gamma_\infty'$ in $\Gamma/[\Gamma, \Gamma]$ is either trivial, finite or isomorphic to $\Z$. Thus we conclude that $ H^1_{\Gamma_\infty'}(\Gamma,\R/\Z )$ is non-trivial as soon as, say $ H^1(\Gamma,\R/\Z )$ is not generated by a single element or $H_\mathrm{cusp}^1(\Gamma, \R)$ is non-trivial.

\subsection{The dimension of cohomology groups}
It is a result of Kazhdan \cite{Kazhdan67} that for discrete, cofinite subgroups of real Lie groups of rank larger than $1$, the abelianization is always torsion. In our case, since $\SO(n+1,1)$ is of rank one, we can however hope to see some free part. In the case of cofinite subgroups $\Gamma\subset \SO(n+1,1)$, the dimension of $H^1(\Gamma, \R)$ (or equivalently the free part of $\Gamma/[\Gamma,\Gamma]$) is not very well understood for arbitrary $n$. The best lower bounds of the rank available in the literature seem to be what follows from the work of Millson \cite{Millson76} and Lubotzky \cite{Lubotzky96}, which gives that any arithmetic subgroup $\Gamma $ (with a few restrictions when $n=3,7$) contains a subgroup such that the dimension of $H^1(\Gamma, \R)$ is at least one. In certain arithmetic situations, we will be able to say more using a connection to automorphic forms.

\subsubsection{Cohomology classes associated to automorphic forms}
Recall the splitting (\ref{ESH}) due to Harder of the cohomology into a cuspidal and an Eisenstein part. We give a brief overview of the description of $H^1_{{\rm cusp}}(\Gamma, \R)$ in terms of automorphic forms, as in \cite{Sarnak90}. We recall the canonical isomorphism between $H^1(\Gamma, \R)$ and the de Rham cohomology group $H^1_{\rm{dR}}(\Gamma\backslash \H^{n+1}, \R)$ consisting of 1-forms. Inside $H^1_{\rm{dR}}(\Gamma\backslash \H^{n+1}, \R)$ we define the subset of cuspidal harmonic 1-forms. 
\begin{defi}
\label{cuspidal 1-form}
A harmonic 1-form $\alpha=f_0 dx_0 +f_1 dx_1 + \cdots + f_n dx_n$
on $\GH$ is a {\bf cuspidal 1-form} if
\begin{enumerate}
    \item $\alpha$ is rapidly decreasing at all cusps of $\Gamma$,
    \item for each cusp $\a$ and $y \geq 0$, we have $$\int_{\mathcal{P}_{\a}}f_{\a,i}(x,y) dx = 0 \ , \quad i=0, \dots, n \ ,$$
where $\sa^{*}\alpha=f_{\a,0} dx_0 +  f_{\a,1} dx_1 \cdots  + f_{\a, n}d x_n$.
\end{enumerate}
\end{defi}
We denote by $\mathrm{Har}^1_{\rm{cusp}}(\GH, \R)$ the space of harmonic cuspidal 1-forms on $\Gamma\backslash \H^{n+1}$. Then we have the following identification
\begin{equation*}
    \mathrm{Har}^1_{\rm{cusp}}(\GH, \R) \cong H^1_{\rm{cusp}}(\Gamma, \R),
\end{equation*}
coming from \cite[(2.14)]{Sarnak90}. This reduces the task of lower bounding the dimension of $H^1_{{\rm cusp}}(\Gamma, \R)$ to constructing cuspidal automorphic forms. For congruence subgroups $\Gamma<\mathrm{SV}_{n-1}$, this can be achieved using certain {\it theta lifts} developed by Shintani \cite{Shintani75} of $\GL_2$ holomorphic forms of weight $(n+1)/2+1$ (for details see \cite[page 21]{Sarnak90}). This gives us non-trivial examples for which Theorem \ref{mainthmHn} applies for any $n$. In the low-dimensional cases $n=1,2$ a lot more can be said, as we will see below.\\

Finally let us see explicitly how to construct a unitary characters from cuspidal automorphic forms. We let
\begin{align*}
    \Phi : \Gamma  &\to H_1 (\Gamma, \Z), \quad
    \gamma \mapsto \{\infty, \gamma \infty \} 
\end{align*}
which induces the canonical isomorphism $H_1(\Gamma, \Z) \cong \Gamma/[\Gamma, \Gamma]$. For $\gamma \in \Gamma$ and $\omega\in \mathrm{Har}^1_{\rm{cusp}}(\GH, \R)$, we define the {\it Poincar\'e pairing}
\begin{equation*}
    \inprod{\gamma}{\omega} := 2 \pi i \int_{\Phi(\Gamma)} \omega = 2 \pi i \int_{P}^{\gamma P} \omega \quad \text{for any } P \in \H^{n+1}.
\end{equation*}
We note that that when $n=1$ and $f$ is a classical Hecke cusp form of weight 2 for $\Gamma$, then $f(z)dz$ is indeed a harmonic cuspidal 1-form on $\Gamma\backslash \H^2$ and the Poincar\'e symbol is equal to (minus) the standard modular symbol (\ref{modularsymbol}):
\begin{equation*}
    \inprod{\gamma}{f(z)dz}=2\pi i\int_{\infty}^{a_{\gamma}/c_{\gamma}} f(z) dz = -\inprod{a_\gamma/c_\gamma}{f}.
\end{equation*}
We observe that if $\gamma \in \Gamma$ is parabolic, then $\inprod{\gamma}{\alpha}=0$. Hence if we define $\chi_{\alpha}(\gamma):=e(\inprod{\gamma}{\alpha})$ then $\chi_{\alpha}$ defines a unitary character trivial on $\Gamma_\infty'$. The kernel of the map $\alpha \mapsto \chi_{\alpha}$ is a full rank lattice $L$ inside $\mathrm{Har}^1_{\rm{cusp}}(\GH, \R)$. If we assume that $\Gamma$ is torsion-free, we indeed obtain the identification $H^1_{{\rm free}}(\Gamma, \R/\Z) \cong  \mathrm{Har}^1_{\rm{cusp}}(\GH, \R) / L$.

\subsubsection{The case of $\H^2$}When $n=1$, we have explicit formulas for the dimensions of both the cuspidal and the Eisenstein part. More precisely we have coming from \cite[Prop. 6.2.3]{Wie19} that 
$$H^1_{{\rm cusp}}(\Gamma, \Z)\cong \R^{2g}, H^1_{{\rm Eis}}(\Gamma, \R)\cong \R^{2(h-1)},$$
where $g$ is the genus and $h$ is the number of inequivalent cusps of the Riemann surface $\Gamma\backslash \H^2$. In particular if $\Gamma=\Gamma_0(N)$ is a standard Hecke congruence subgroup, we know that $g\sim \frac{N\cdot\prod_{p|N}(1+p^{-1})}{12}$ and $h=\sum_{d|N}\varphi(d,N/d)$ and we conclude that we can find towers of Hecke congruence subgroups such that both the cuspidal and Eisenstein part goes to infinity.

\subsubsection{The case of $\H^3$}
When $n=2$ there has been a lot of activity recently and we refer to the survey of \c{S}eng\"{u}n \cite{Sengun14} for an excellent and more thorough overview. In this case no formulas are known in general for the ranks of the cuspidal and Eisenstein part and the best one can hope for are lower bounds.

Regarding the Eisenstein part, we can describe it explicitly when $\Gamma$ is torsion-free. In this case, we have that $H^1_{{\rm Eis}}(\Gamma, \R)\cong \R^{h}$, where $h$ is the number of cusps of $\Gamma\backslash \H^3$, see \cite[Proposition 7.5.6]{ElGrMe98}. The same conclusion holds for co-finite subgroups $\Gamma\leq \SL_2(\mathcal{O}_D)$, where $\mathcal{O}_D$ is the ring of integers of the imaginary quadratic field $\Q(\sqrt{D})$ with $D<0$ a fundamental discriminant not equal to $-4,-3$ (in which case there might be torsion in $\Gamma$). In the case of co-finite subgroups $\Gamma\leq \SL_2(\mathcal{O_D})$ with $D=-4,-3$ the picture is much more mysterious, but a lot of numerics are available in \cite{Sengun11} and \cite[Ch. 7.5]{ElGrMe98}. 

For the cuspidal part there are some useful results giving lower bounds on the rank. First of all Rohlfs \cite{rohlfs} showed that 
$$\dim H^1_{{\rm cusp}}(\SL_2(\mathcal{O}_D), \R)\geq  \frac{\varphi(D)}{6}-\frac{1}{2}-h(D), $$
where $h(D)$ denotes the class number of $\Q(\sqrt{D})$. Furthermore \c{S}eng\"{u}n and Turkelli \cite{SeTu16} proved that if $D$ is a fundamental discriminant such that $h(D)=1$, $p$ is a rational prime which is inert in $\Q(\sqrt{D})$ and $\Gamma_0(p^n)\subset \SL_2(\mathcal{O}_D)$ is a congruence subgroup, then we have 
$$ \dim H^1_{{\rm cusp}}(\Gamma_0(p^n), \R)\geq p^{6n},  $$
as $n\rightarrow \infty$ (an upper bound of $p^{10n}$ has been proved by Calegari and Emerton \cite{CaEm09}). In the case of cocompact groups stronger results were obtained by Kionke and Schwermer \cite{KiSc15}. 

\subsection{Torsion in the (co)homology of arithmetic groups}
Now we will discuss what is known about the torsion part of $H_1(\Gamma,\Z)$ when $\Gamma\subset \SO(n+1,1)$ is a cofinite, arithmetic subgroup. In the simplest case $n=1$, we know that all the torsion in the abeliazation comes from the torsion in the subgroup itself and thus in particular $\Gamma/[\Gamma,\Gamma]$ is torsion-free when $\Gamma$ is so.

It was noticed a long time ago in unpublished work by Grunewald and Mennicke that in the case $n=2$ there is a lot of torsion in the abeliazation of congruence subgroups. See  \c{S}eng\"{u}n's work \cite{Sengun11} for some recent extensive computations.

The study of torsion in the abelianization of $\Gamma$ fits into a more general framework of understanding the torsion in the homology of arithmetic groups as in the work of Bergeron and Venkatesh \cite{BergeronVenkatesh13}.  Bergeron and Venkatesh have conjectured that when $\Gamma$ is a congruence subgroup of $\SL_2(\mathcal{O}_{D})$ with $D<0$ a negative fundamental discriminant, then the torsion in $\Gamma/[\Gamma,\Gamma]$ grows exponentially with the index $[\SL_2(\mathcal{O}_{D}):\Gamma]$. 

More generally the conjectures predicts that the torsion in the cohomology of symmetric spaces associated to a semisimple Lie group $G$ will grow exponentially in towers of congruence subgroups exactly if we consider the middle dimensional cohomology and if the {\it fundamental rank} (or {\lq\lq}deficiency{\rq\rq}) $\delta(G):=\rank (G)-\rank (K)$ is $1$ (here $K$ is a maximal compact). It follows from \cite[1.2]{BergeronVenkatesh13} that the fundamental rank of $\SO(n+1,1)$ is equal to $1$ exactly if $n$ is even. And thus we see that we will have exponential growth of the torsion of $\Gamma/[\Gamma,\Gamma]$ when $\Gamma$ runs through a tower of congruence groups exactly when $n=2$ (corresponding to Kleinian groups).

For $n>2$ the torsion should conjecturally {\it not} grow exponentially, but there might still be torsion, which is equally arithmetically interesting in view of \cite{Scholze15}. There seems however to be no experimental or theoretical work available in this case.

\bibliographystyle{plain}

\end{document}